\documentclass[11pt]{amsart}
\usepackage[amssymb,xy,defs,ams,white,spread]{tms}

\title[Extremal conformal structures]{Extremal Conformal Structures\\ on Projective Surfaces}
\author[T.~Mettler]{Thomas Mettler}
\address{Institut f\"ur Mathematik, Goethe-Universit\"at Frankfurt, 60325 Frankfurt am Main, Germany}
\email{mettler@math.uni-frankfurt.de,mettler@math.ch}
\date{December 18, 2018}
\thanks{A part of the research for this article was carried out while the author was visiting the Mathematical Institute at the University of Oxford on a postdoctoral fellowship of the Swiss NSF PA00P2\_142053. The author would like to thank the Mathematical Institute for its hospitality.}
\keywords{projective structures, Weyl connections, Hitchin component, moving frames}
\newcommand{\rot}{r\e^{\i\phi}}
\newcommand{\con}{{}^{[g]}\nabla}
\newcommand{\M}{M}
\newcommand{\Ym}{h_{\mathfrak{p}}}
\newcommand{\F}{F}
\newcommand{\projb}{P}

\subjclass[2010]{53A20, 53C28, 58E20}

\begin{document}

\begin{abstract}
We introduce a new functional $\mathcal{E}_{\mathfrak{p}}$ on the space of conformal structures on an oriented projective manifold $(M,\mathfrak{p})$. The nonnegative quantity $\mathcal{E}_{\mathfrak{p}}([g])$ measures how much $\mathfrak{p}$ deviates from being defined by a $[g]$-conformal connection. In the case of a projective surface $(\Sigma,\mathfrak{p})$, we canonically construct an indefinite K\"ahler--Einstein structure $(h_{\mathfrak{p}},\Omega_{\mathfrak{p}})$ on the total space $Y$ of a fibre bundle over $\Sigma$ and show that a conformal structure $[g]$ is a critical point for $\mathcal{E}_{\mathfrak{p}}$ if and only if a certain lift $\widetilde{[g]} : (\Sigma,[g]) \to (Y,h_{\mathfrak{p}})$ is weakly conformal. In fact, in the compact case $\mathcal{E}_{\mathfrak{p}}([g])$ is -- up to a topological constant -- just the Dirichlet energy of $\widetilde{[g]}$. As an application, we prove a novel characterisation of properly convex projective structures among all flat projective structures. As a by-product, we obtain a Gauss--Bonnet type identity for oriented projective surfaces.        
\end{abstract}

\maketitle

\tableofcontents

\section{Introduction}

A~\textit{projective structure} on an $n$-manifold $M$ is an equivalence class $\mathfrak{p}$ of torsion-free connections on the tangent bundle $TM$, where two connections are called projectively equivalent if they share the same unparametrised geodesics. A manifold $M$ equipped with a projective structure $\mathfrak{p}$ will be called a~\textit{projective manifold}. A~\textit{conformal structure} on $M$ is an equivalence class $[g]$ of Riemannian metrics on $M$, where two metrics are called conformally equivalent if they differ by a scale factor. Naively, one might think of projective and conformal structures as formally similar, since both arise by defining a notion of equivalence on a geometric structure. However, the formal similarity is more substantial. For instance, Kobayashi has shown~\cite{MR0355886} that both projective -- and conformal structures admit a treatment as Cartan geometries with $|1|$-graded Lie algebras. Here we exploit the fact that both structures give rise to affine subspaces modelled on $\Omega^1(M)$ of the infinite-dimensional affine space $\mathfrak{A}(M)$ of torsion-free connections on $M$. Indeed, it is a classical result due to Weyl~\cite{zbMATH02603060} that two torsion-free connections on $TM$ are projectively equivalent if and only if their difference -- thought of as a section of $S^2(T^*M)\otimes TM$ -- is pure trace. Consequently, the representative connections of a projective structure $\mathfrak{p}$ on $M$ define an affine subspace $\mathfrak{A}_{\mathfrak{p}}(M)$ which is modelled on $\Omega^1(M)$. Moreover, it follows from Koszul's identity, that the torsion-free connections preserving a conformal structure $[g]$ on $M$ are of the form
$$
{}^g\nabla+g\otimes \beta^{\sharp}-\beta\otimes\mathrm{Id}-\mathrm{Id}\otimes\beta,
$$
with $g \in [g]$, $\beta \in \Omega^1(M)$ and where ${}^g\nabla$ denotes the Levi-Civita connection of $g$. Hence, the space of torsion-free $[g]$-conformal connections on $TM$ is an affine subspace $\mathfrak{A}_{[g]}(M)$ modelled on $\Omega^1(M)$ as well. It is an elementary computation to check that if $\mathfrak{A}_{[g]}(M)$ and $\mathfrak{A}_{\mathfrak{p}}(M)$ intersect, then they do so in a unique point. Therefore, we may ask if in general one can distinguish a point in $\mathfrak{A}_{\mathfrak{p}}(M)$ and a point in $\mathfrak{A}_{[g]}(M)$ which are `as close as possible'. This is indeed the case. More precisely, we show that the choice of a conformal structure $[g]$ on $(M,\mathfrak{p})$ determines a $1$-form $A_{[g]}$ on $M$ with values in the endomorphisms of $TM$, as well as a unique $[g]$-conformal connection ${}^{[g]}\nabla \in \mathfrak{A}_{[g]}(M)$ so that ${}^{[g]}\nabla+A_{[g]} \in \mathfrak{A}_{\mathfrak{p}}(M)$. The $1$-form $A_{[g]}$ appeared previously in the work of Matveev \& Trautman~\cite{MR3212870} and may be thought of as the `difference' between $\mathfrak{p}$ and $[g]$. In particular, if $M$ is oriented, we obtain a $\mathrm{Diff}(M)$-invariant functional
$$
\mathcal{F}(\mathfrak{p},[g])=\int_{M}|A_{[g]}|^n_g d \mu_g.
$$
Fixing a projective structure $\mathfrak{p}$ on $M$, we may consider the functional $\mathcal{E}_{\mathfrak{p}}=\mathcal{F}(\mathfrak{p},\cdot)$, which is a functional on the space $\mathfrak{C}(M)$ of conformal structures on $M$ only. It is natural to study the infimum of $\mathcal{E}_{\mathfrak{p}}$ among all conformal structures on $M$, and to ask whether there is actually a minimising conformal structure which achieves this infimum. This infimum -- which may be considered as a measure of how far $\mathfrak{p}$ deviates from being defined by a conformal connection -- is a new global invariant for oriented projective manifolds. 

Of particular interest is the case of surfaces where $\mathcal{E}_{\mathfrak{p}}$ is just the square of the $L^2$-norm of $A_{[g]}$ taken with respect to $[g]$ and this is the case that we study in detail in this article. It turns out that in the surface case the functional $\mathcal{E}_{\mathfrak{p}}$ also arises from a rather different viewpoint, which simplifies the computation of its variational equations by using the technique of moving frames.    

Inspired by the twistorial construction of holomorphic projective structures by Hitchin \cite{MR699802}, it was shown in~\cite{MR728412},~\cite{MR812312} how to construct a `twistor space` for smooth projective structures. The choice of a projective structure $\mathfrak{p}$ on an oriented surface $\Sigma$ induces a complex structure on the total space of the disk bundle $Z \to \Sigma$ whose sections are conformal structures on $\Sigma$. In this sense, $\mathcal{E}_{\mathfrak{p}}([g])$ can be interpreted as measuring the failure of $[g](\Sigma) \subset Z$ to be a holomorphic curve in $Z$. We proceed to show that $\mathfrak{p}$ canonically defines an indefinite K\"ahler-Einstein structure $(h_{\mathfrak{p}},\Omega_{\mathfrak{p}})$ on a certain submanifold $Y$ of the projectivised holomorphic cotangent bundle $\mathbb{P}(T^*_{\C}Z^{1,0})$ of $Z$. Moreover, every conformal structure $[g] : \Sigma \to Z$ admits a lift $\widetilde{[g]} : \Sigma \to Y$ so that the variational equations can be expressed as follows:
\setcounter{section}{4}
\setcounter{thm}{7}

\begin{thm}
Let $(\Sigma,\mathfrak{p})$ be an oriented projective surface. A conformal structure $[g]$ on $\Sigma$ is extremal for $\mathfrak{p}$ if and only if $\widetilde{[g]} : (\Sigma,[g]) \to (Y,\Ym)$ is weakly conformal.  
\end{thm}
Here we say that $[g]$ is~\textit{extremal} for $\mathfrak{p}$ if it is a critical point of $\mathcal{E}_{\mathfrak{p}}$ with respect to compactly supported variations. Moreover, by weakly conformal we mean that there exists a smooth (and possibly vanishing) function $f$ on $\Sigma$ so that for some -- and hence any -- representative metric $g \in [g]$, we have $\widetilde{[g]}^*h_{\mathfrak{p}}=fg$. In fact, in the compact case $\mathcal{E}_{\mathfrak{p}}([g])$ is, up to the topological constant $-2\pi\chi(\Sigma)$, just the Dirichlet energy of $\widetilde{[g]}$. As a consequence, we obtain an optimal lower bound:
\setcounter{thm}{9}
\begin{thm}
Let $(\Sigma,\mathfrak{p})$ be a compact oriented projective surface. Then for every conformal structure $[g] : \Sigma \to Z$ we have
$$
\frac{1}{2}\int_{\Sigma}\tr_g \widetilde{[g]}^*\Ym\, d\mu_g\geqslant -2\pi\chi(\Sigma),
$$
with equality if and only if $\mathfrak{p}$ is defined by a $[g]$-conformal connection.\end{thm}   

We then turn to the problem of finding non-trivial examples of  projective structures for which $\mathcal{E}_{\mathfrak{p}}$ admits extremal conformal structures. The conformal connection ${}^{[g]}\nabla$ determined by the choice of a conformal structure $[g]$ on $(\Sigma,\mathfrak{p})$ may equivalently be thought of as a torsion-free connection $\varphi$ on the principal $\mathrm{GL}(1,\C)$-bundle of complex linear coframes of $(\Sigma,[g])$. In addition, the $1$-form $A_{[g]}$ turns out to be twice the real part of a section $\alpha$ of $K_{\Sigma}^2\otimes \ov{K_{\Sigma}^{*}}$, where $K_{\Sigma}$ denotes the canonical bundle of $(\Sigma,[g])$. We provide another interpretation of the variational equations by proving that $[g]$ is extremal for $\mathfrak{p}$ if and only if the quadratic differential $\nabla_{\varphi}^{\prime\prime}\alpha$ vanishes identically. Here $\nabla_{\varphi}$ denotes the connection induced by $\varphi$ on $K_{\Sigma}^2\otimes \ov{K_{\Sigma}^{*}}$ and $\nabla^{\prime\prime}_{\varphi}$ its $(0,\! 1)$-part. Applying the Riemann--Roch theorem, it follows that a projective structure $\mathfrak{p}$ on the $2$-sphere $S^2$ admits an extremal conformal structure if and only if $\mathfrak{p}$ is defined by a conformal connection. 

While there are no non-trivial critical points for projective structures on the $2$-sphere, the situation is quite different for surfaces with negative Euler characteristic. Indeed, the condition of having a vanishing quadratic differential appeared previously in the projective differential geometry literature. In the celebrated paper ``\textit{Lie groups and Teichm\"uller space}''~\cite{MR1174252} Hitchin proposed a generalisation of Teichm\"uller space $\mathcal{H}_2$ by identifying a connected component $\mathcal{H}_n$ -- nowadays called the~\textit{Hitchin component} -- in the space of conjugacy classes
of representations of $\pi_1(\Sigma)$ into $\mathrm{PSL}(n,\R$).\footnote{More generally, representation into a real split simple Lie group.} Here $\Sigma$ denotes a compact oriented surface whose genus exceeds one. Using the theory of Higgs bundles~\cite{MR887284} and harmonic map techniques, Hitchin showed that the choice of a conformal structure $[g]$ on $\Sigma$ gives an identification
$$
\mathcal{H}_n\simeq \bigoplus_{\ell=2}^n H^0(\Sigma,K_{\Sigma}^{\ell}).
$$
Hitchin conjectured that $\mathcal{H}_3$ is the space of conjugacy classes of monodromy representations of (flat) properly convex projective structures, a fact later confirmed by Choi and Goldman~\cite{MR1145415} (the geometric interpretation of the Hitchin component for $n>3$ is a topic of current interest, c.f.~\cite{MR2981818},~\cite{MR3525096},~\cite{MR2373231} for recent results). Teichm\"uller space being para\-metri\-sed by holomorphic quadratic differentials, one might ask if there is a unique choice of a conformal structure on $\Sigma$, so that $\mathcal{H}_3$ is parametrised in terms of cubic holomorphic differentials only. This is indeed the case, as was shown independently by Labourie~\cite{MR2402597} and Loftin~\cite{MR1828223} (see also~\cite{MR3039771} and~\cite{MR3432157} for recent work treating the non-compact case and the case of convex polygons, as well as~\cite{MR3583351} treating the case of a general real split rank $2$ group). Furthermore, the conformal structure $[g]$ making the quadratic differential vanish is the conformal equivalence class of the so-called~\textit{Blaschke metric}, which arises by realising the universal cover of a properly convex projective surface as a complete hyperbolic affine $2$-sphere, see in particular~\cite{MR1828223}.

Calling a conformal structure $[g]$ on $(\Sigma,\mathfrak{p})$~\textit{closed}, if $\varphi$ induces a flat connection on $\Lambda^2(T^*\Sigma)$, we obtain a novel characterisation of properly convex projective structures among flat projective structures:
\setcounter{section}{5}
\setcounter{thm}{1}
\begin{thm}
Let $(\Sigma,\mathfrak{p})$ be a compact oriented flat projective surface of negative Euler characteristic. Suppose $\mathfrak{p}$ is properly convex, then the conformal equivalence class of the Blaschke metric is closed and extremal for $\mathcal{E}_{\mathfrak{p}}$. Conversely, if $\mathcal{E}_{\mathfrak{p}}$ admits a closed extremal conformal structure $[g]$, then $\mathfrak{p}$ is properly convex and $[g]$ is the conformal equivalence class of the Blaschke metric of $\mathfrak{p}$.       
\end{thm}

We conclude with some remarks about the possible relation between our functional and the energy functional on Teichm\"uller space~\cite{MR887285},~\cite{MR2482204} which one can associate to a representation in the Hitchin component. Finally, as a by-product of our ideas, we obtain a Gauss--Bonnet type identity for oriented projective surfaces, which we briefly discuss in Appendix~\ref{Ap:GB}. 

\bigskip
\noindent\textbf{Acknowledgements.} I would like to thank Nigel Hitchin, Gabriel Paternain, Maciej Dunajski, Charles Frances, Karin Melnick and Stefan Rosemann for helpful conversations or correspondence regarding the topic of this article. The author is also grateful to the anonymous referee for helpful remarks. 

\setcounter{section}{1}

\section{Projective and conformal structures}

\subsection{Preliminaries}

Throughout the article, all manifolds are assumed to be connected, have empty boundary and unless stated otherwise, all manifolds and maps are assumed to be smooth, i.e., $C^{\infty}$. Also, we adhere to the convention of summing over repeated indices.  

\subsubsection{Notation}

For $\mathbb{F}=\mathbb{R},\mathbb{C}$ the field of real or complex numbers, we denote by $\mathbb{F}^n$ the space of column vectors of height $n$ and by $\mathbb{F}_n$ the space of row vectors of length $n$ whose entries are elements of $\mathbb{F}$. Also, we denote by $\mathbb{FP}^2=\left(\mathbb{F}^3\setminus\{0\}\right)/\mathbb{F}^*$ the space of one-dimensional linear subspaces in $\mathbb{F}^3$, that is, the real or complex projective plane. We denote by $\mathbb{S}^2=\left(\R^3\setminus\{0\}\right)/\R^+$ the space of oriented one-dimensional linear subspaces in $\R^3$, that is, the projective $2$-sphere. Likewise, we write $\mathbb{FP}_{2}=\left(\mathbb{F}_3\setminus\{0\}\right)/\mathbb{F}^*$ for the dual (real or complex) projective plane and $\mathbb{S}_2=\left(\R_3\setminus\{0\}\right)/\R^+$ for the dual projective $2$-sphere. For a non-zero vector $x \in \mathbb{F}^3$ we write $[x]$ for its corresponding point in $\mathbb{FP}^2$ and for a non-zero vector $\xi \in \mathbb{F}_3$ we write $[\xi]$ for its corresponding point in $\mathbb{FP}_2$. For non-zero vectors $x\in\R^3$ and $\xi \in \R_3$ we also use the notation $[x]_+$ and $[\xi]_+$ to denote the corresponding points in $\mathbb{S}^2$ and $\mathbb{S}_2$. Finally, we use the notation $F(\mathbb{F}_3)$ to denote the space of complete flags in $\mathbb{F}_3$ whose points are pairs $(\ell,\Pi)$ with $\Pi$ being an $\mathbb{F}$ two-dimensional linear subspace of $\mathbb{F}_3$ containing the line $\ell$.

\subsubsection{The coframe bundle} Recall that the~\textit{coframe bundle} of an $n$-ma\-ni\-fold $M$ is the bundle $\upsilon : F(T^*M) \to \M$ whose fibre at a point $p \in M$ consists of the linear isomorphisms $u : T_pM \to \R^n$. The group $\mathrm{GL}(n,\R)$ acts transitively from the right on each $\upsilon$-fibre by the rule $R_a(u)=u\cdot a =a^{-1}\circ u$ for all $a \in \mathrm{GL}(n,\R)$. This action turns $\upsilon : F(T^*M) \to M$ into a principal right $\mathrm{GL}(n,\R)$-bundle. The coframe bundle is equipped with a tautological $\R^n$-valued $1$-form $\omega=(\omega^i)$ defined by $\omega_u=u\circ \upsilon^{\prime}_u$. Note that $\omega$ satisfies the equivariance property $R_a^*\omega=a^{-1}\omega$ for all $a \in \mathrm{GL}(n,\R)$. The exterior derivative of local coordinates $x : U \to \R^n$ on $M$ defines a natural section $\tilde{x} : U \to F(T^*M)$ having the reproducing property $\tilde{x}^*\omega=\d x$. We will henceforth write $F$ instead of $F(T^*M)$ whenever $M$ is clear from the context.

\subsubsection{Associated bundles} Throughout the article we will frequently make use of the notion of an associated bundle of a principal bundle. The reader will recall that if $\pi : P \to M$ is a principal right $\mathrm{G}$-bundle and $(\rho,N)$ a pair consisting of a manifold $N$ and a homomorphism $\rho : \mathrm{G} \to \mathrm{Diff}(N)$ into the diffeomorphism group of $N$, then we obtain an associated fibre bundle with typical fibre $N$ and structure group $\mathrm{G}$ whose total space is $P\times_{\rho} N$, that is, the elements of $P\times_{\rho} N$ are pairs $(u,p)$ subject to the equivalence relation
$$
(u_1,p_1)\sim (u_2,p_2) \iff u_2=u_1\cdot g, \quad p_2=\rho(g^{-1})(p_1), \quad g \in \mathrm{G}. 
$$
A section $s$ of $P\times_{\rho} N$ is then given by a map $\sigma_{s} : P \to N$ which is equivariant with respect to the $\mathrm{G}$-right action on $P$ and the right action of $\mathrm{G}$ on $N$ induced by $\rho$. We say that $s$ is~\textit{represented by} $\sigma_{s}$. If $N$ is an affine/linear space and the $\mathrm{G}$-action induced by $\rho$ is affine/linear, then the associated bundle is an affine/vector bundle. 

\subsection{Projective structures} Recall that the set $\mathfrak{A}(\M)$ of torsion-free connections on the tangent bundle of an $n$-manifold $\M$ is the space of sections of an affine bundle $\mathrm{A}(\M) \to \M$ of rank $\frac{1}{2}n^2(n+1)$ which is modelled on the vector bundle $V=S^2(T^*M)\otimes TM$. We have a canonical trace mapping $\tr : V \to T^*M$ as well as an inclusion
$$
\iota : T^*M \to V, \quad \nu \mapsto \nu\otimes\mathrm{Id}+\mathrm{Id}\otimes \nu. 
$$
For every $v \in V$ we let $v_0$ denote its trace-free part, so that
$$
v_0=v-\frac{1}{(n+1)}\iota(\tr v).
$$
A projective structure $\mathfrak{p}$ on a manifold $M$ of dimension $n>1$ is an equivalence class of torsion-free connections on $TM$, where two connections are declared to be equivalent if they share the same unparametrised geodesics. Weyl~\cite{zbMATH02603060} observed the following: 
\begin{lem}\label{Weylfund}
Two torsion-free connections $\nabla$ and $\nabla^{\prime}$ on $T\M$ are projectively equivalent if and only if $(\nabla-\nabla^{\prime})_0=0$.
\end{lem}
Consequently, the set $\mathfrak{P}(\M)$ of projective structures on $\M$ is the space of sections of an affine bundle $\mathrm{P}(M) \to M$ of rank $\frac{1}{2}(n+2)n(n-1)$ which is modelled on the traceless part $V_0$ of the vector bundle $V$. We will use the notation $\mathfrak{p}(\nabla)$ for the projective structure $\mathfrak{p}$ that is defined by a connection $\nabla$. A consequence of Weyl's result is that the set of representative connections of a projective structure $\mathfrak{p}$ is an affine subspace $\mathfrak{A}_{\mathfrak{p}}(M)\subset \mathfrak{A}(M)$ of the space of torsion-free connections which is modelled on the space of $1$-forms on $M$. 

\subsection{Conformal structures}
A conformal structure on a manifold $M$ of dimension $n>1$ is an equivalence class $[g]$ of Riemannian metrics on $M$, where two metrics $g$ and $\hat{g}$ are declared to be equivalent if there exists a smooth   function $f$ on $M$ so that $\hat{g}=\e^{2f} g$. Equivalently, a conformal structure $[g]$ on $M$ is a (smooth) choice of a coframe for every point $p$ in $M$, well defined up to orthogonal transformation and scaling. Consequently, the set $\mathfrak{C}(M)$ of conformal structures on $M$ is the space of sections of $\mathrm{C}(M)=F/\left(\R^+\times\mathrm{O}(n)\right) \to M$, where $\R^+\times \mathrm{O}(n)$ is the subgroup of $\mathrm{GL}(n,\R)$ consisting of matrices $a$ having the property that $aa^t$ is a non-zero multiple of the identity matrix. 

A torsion-free connection $\nabla$ on $TM$ is called a~\textit{Weyl connection} or~\textit{conformal connection} for the conformal structure $[g]$ on $M$ if the parallel transport maps of $\nabla$ are angle-preserving with respect to $[g]$. A torsion-free connection $\nabla$ is $[g]$-conformal if for some (and hence any) representative metric $g \in [g]$ there exists a $1$-form $\beta$ on $M$ such that
$$
\nabla g= 2 \beta \otimes g. 
$$
It is a simple consequence of Koszul's identity that the $[g]$-conformal connections are of the form
\begin{equation}\label{confcon}
{}^{(g,\beta)}\nabla={}^g\nabla+g\otimes \beta^{\sharp}-\beta\otimes\mathrm{Id}-\mathrm{Id}\otimes \beta, 
\end{equation}
where $g\in [g]$, $\beta$ is a $1$-form on $M$ with $g$-dual vector field $\beta^{\sharp}$ and ${}^g\nabla$ denotes the Levi-Civita connection of $g$. Consequently, the set of $[g]$-conformal connections defines an affine subspace $\mathfrak{A}_{[g]}(M)\subset \mathfrak{A}(M)$ which is modelled on the space of $1$-forms on $M$ as well. For later usage we also record that for every smooth function $f$ on $M$ we have
$$
{}^{(\exp(2f)g,\beta +\d f)}\nabla={}^{(g,\beta)}\nabla,
$$
as the reader may easily verify using the identity~\cite[Theorem 1.159]{MR867684}
\begin{equation}\label{eq:confchangeLC}
{}^{\exp(2f)g}\nabla={}^g\nabla-g\otimes{}^g\nabla f+\iota(\d f).
\end{equation}
In particular, if $\beta$ is exact, so that $\beta=\d f$ for some smooth function $f$ on $M$, then ${}^{(g,\beta)}\nabla={}^{\exp(-2f)g}\nabla$ and hence the conformal connection determined by $(g,\beta)$ is the Levi-Civita connection of the metric $\e^{-2f}g$. 

We also use the notation ${}^{[g]}\nabla$ for a connection preserving the conformal structure $[g]$.

\subsection{Compatibility of projective and conformal structures}

Since both projective -- and conformal structures give rise to affine subspaces of $\mathfrak{A}(M)$ of the same type, we may ask how two such spaces intersect. 

\begin{lem}\label{intersect}
Let $[g]$ be a conformal -- and $\mathfrak{p}$ a projective structure on $M$. Then $\mathfrak{A}_{[g]}(M)$ and $\mathfrak{A}_{\mathfrak{p}}(M)$ intersect in at most one point. 
\end{lem}
\begin{proof}
Suppose the $[g]$-conformal connections $\con$ and ${}^{[g]}\hat{\nabla}$ are elements in $\mathfrak{A}_{\mathfrak{p}}(M)$. Then, by Lemma~\ref{Weylfund}, there exists a $1$-form $\Upsilon$ on $M$ so that
$$
\con={}^{[g]}\hat{\nabla}+\iota(\Upsilon). 
$$
Fixing a Riemannian metric $g$ defining $[g]$, we also have $1$-forms $\beta,\hat{\beta}$ on $M$ so that
$$
\con={}^g\nabla+g\otimes \beta^{\sharp}-\iota(\beta)\quad \text{and}\quad{}^{[g]}\hat{\nabla}={}^g\nabla+g\otimes \hat{\beta}^{\sharp}-\iota(\hat{\beta}).
$$
Applying these formulae we obtain
$$
\iota(\Upsilon+\beta-\hat{\beta})=g\otimes\left(\beta^{\sharp}-\hat{\beta}^{\sharp}\right).
$$
Taking the trace gives 
$$
(n+1)\left(\Upsilon+\beta-\hat{\beta}\right)=\beta-\hat{\beta}, 
$$
so that $\Upsilon=-\frac{n}{(n+1)}(\beta-\hat{\beta})$. Therefore we must have
$$
\iota\left(\beta-\hat{\beta}\right)=(n+1)g\otimes \left(\beta^{\sharp}-\hat{\beta}^{\sharp}\right).
$$
Contracting this last equation with the dual metric $g^\sharp$ implies
$$
0=(n+2)(n-1)\left(\beta^{\sharp}-\hat{\beta}^{\sharp}\right), 
$$
so that $\beta=\hat{\beta}$ provided $n>1$. It follows that $\Upsilon$ vanishes too, therefore $\con={}^{[g]}\hat{\nabla}$, as claimed.  
\end{proof}
\begin{rmk}\label{rmk:diffequimap}
Lemma~\eqref{intersect} raises the question whether or not one can still determine a unique point $\con \in \mathfrak{A}_{[g]}(M)$ and a unique point $\nabla \in \mathfrak{A}_{\mathfrak{p}}(M)$ in the general case, where $\mathfrak{A}_{[g]}(M)$ and $\mathfrak{A}_{\mathfrak{p}}(M)$ might not intersect. Formally speaking, we are interested in maps
$$
\psi=\left(\psi^1,\psi^2\right) : \mathfrak{P}(M) \times \mathfrak{C}(M) \to \mathfrak{A}(M)\times \mathfrak{A}(M) 
$$ satisfying the following properties:
\begin{itemize}
\item[(i)] $\psi^1(\mathfrak{p},[g]) \in \mathfrak{A}_{\mathfrak{p}}(M)$ and $\psi^2(\mathfrak{p},[g]) \in \mathfrak{A}_{[g]}(M)$;
\item[(ii)] If $\mathfrak{A}_{\mathfrak{p}}(M)\cap \mathfrak{A}_{[g]}(M)$ is non-empty, then $\psi^2(\mathfrak{p},[g])-\psi^1(\mathfrak{p},[g])=0$;
\item[(iii)] $\psi$ is equivariant with respect to the natural right action of the diffeomorphism group $\mathrm{Diff}(M)$ on $\mathfrak{P}(M)\times \mathfrak{C}(M)$ and $\mathfrak{A}(M)\times \mathfrak{A}(M)$. 
\end{itemize}
We will next discuss a geometrically natural and explicit map $\psi$ having these properties.
\end{rmk}
To this end let $g$ be a Riemannian metric on $M$ and $\nabla$ a torsion-free connection on $TM$. Consider the first-order differential operator for $g$ mapping into the space of $1$-forms on $M$ with values in $\mathrm{End}(TM)$ 
\begin{equation}\label{defalphag}
g \mapsto A_{[g]}=\left(\nabla-{}^g\nabla-g\otimes X_g\right)_0,
\end{equation}
where $X_g \in \Gamma(TM)$ is
\begin{equation}\label{defvect}
X_g=\frac{(n+1)}{(n+2)(n-1)}\,\tr\left(g^{\sharp}\otimes (\nabla-{}^g\nabla)_0\right).
\end{equation}
The following result is essentially contained in~\cite{MR3212870} -- except for (vi). For the convenience of the reader we include a proof.
\begin{thm}[Matveev \& Trautman,~\cite{MR3212870}]\label{quasilinop1}
The $1$-form $A_{[g]}$ has the following properties:
\begin{itemize}
\item[(i)] the endomorphism $A_{[g]}(X)$ is trace-free for all $X \in \Gamma(TM)$;
\item[(ii)] for all $X,Y \in \Gamma(TM)$ we have $A_{[g]}(X)Y=A_{[g]}(Y)X$;
\item[(iii)] $A_{[g]}$ only depends on the projective equivalence class of $\nabla$; 
\item[(iv)] $A_{[g]}$ only depends on the conformal equivalence class of $g$;
\item[(v)] $A_{[g]}\equiv 0$ if and only if there exists a $[g]$-conformal connection which is projectively equivalent to $\nabla$;
\item[(vi)] for $n=2$ the endomorphism $A_{[g]}(X)$ is symmetric with respect to $[g]$ for all $X \in \Gamma(TM)$;
\end{itemize}
\end{thm}

\begin{proof} The properties (i) and (ii) are obvious from the definition. 

(iii) Recall that two affine torsion-free connections $\nabla$ and $\hat{\nabla}$ are projectively equivalent if and only if $(\nabla-\hat{\nabla})_0=0$. The claim follows from the linearity of the ``taking the trace-free part'' operation. 

(iv) Let $\hat{g}=\e^{2f}g$ for some smooth real-valued function $f$ on $M$. Then we have
$$
{}^{\hat{g}}\nabla={}^g\nabla-g\otimes{}^g\nabla f+\iota(\d f)
$$
and hence
\begin{align*}
\left(\nabla-{}^{\hat{g}}\nabla\right)_0&=\left(\nabla-{}^g\nabla\right)_0+\left(g\otimes {}^g\nabla f-\iota(\d f)\right)_0\\
&=\left(\nabla-{}^g\nabla\right)_0+\left(g\otimes {}^g\nabla f\right)_0\\
&=\left(\nabla-{}^g\nabla\right)_0+g\otimes {}^g\nabla f-\frac{1}{(n+1)}\iota(\d f).
\end{align*}
We obtain
\begin{align*}
X_{\hat{g}}&=\frac{(n+1)}{(n+2)(n-1)}\tr\left[\hat{g}^{\sharp}\otimes \left(\left(\nabla-{}^g\nabla\right)_0+g\otimes {}^g\nabla f-\frac{1}{(n+1)}\iota(\d f)\right)\right]\\
&=\e^{-2f}\left(X_g+\frac{n(n+1)}{(n+2)(n-1)}{}^g\nabla f-\frac{2}{(n+2)(n-1)}{}^g\nabla f\right)\\
&=\e^{-2f}\left(X_g+{}^g\nabla f\right).
\end{align*}
This gives
\begin{align*}
{}^{\hat{g}}\nabla+\hat{g}\otimes X_{\hat{g}}&={}^g\nabla-g\otimes{}^g\nabla f+\iota(\d f)+\e^{2f}g\otimes\e^{-2f}\left(X_g+{}^g\nabla f\right)\\
&={}^g\nabla+g\otimes X_g+\iota(\d f),
\end{align*}
so that
$$
\left({}^{\hat{g}}\nabla+\hat{g}\otimes X_{\hat{g}}\right)_0=\left({}^g\nabla+g\otimes X_g\right)_0,
$$
which shows that $A_{[g]}$ does indeed only depend on the conformal class of $g$. 

(v) Recall that the ${[g]}$-conformal connections are of the form
$$
\con={}^g\nabla+g\otimes \beta^{\sharp}-\iota(\beta),
$$
where $g$ is any metric in the conformal class ${[g]}$ and $\beta$ is some $1$-form on $M$. Therefore we have 
$$
\left(\con-{}^{g}\nabla\right)_0=\left(g\otimes\beta^{\sharp}\right)_0=g\otimes \beta^{\sharp}-\frac{1}{(n+1)}\iota (\beta)
$$
and thus as before we compute that $X_g=\beta^{\sharp}$. We obtain
\begin{align*}
A_{[g]}&=\left[{}^{[g]}\nabla-\left({}^g\nabla+g\otimes X_g\right)\right]_0\\
&=\left[{}^g\nabla+g\otimes \beta^{\sharp}-\iota(\beta)-{}^g\nabla-g\otimes \beta^{\sharp}\right]_0=\left[-\iota(\beta)\right]_0=0.
\end{align*}
Conversely, suppose $\mathfrak{p}$ is a projective structure for which there exists a conformal structure ${[g]}$ with $A_{[g]}\equiv 0$. Fixing a Riemannian metric $g \in [g]$ and a $\mathfrak{p}$-representative connection $\nabla$, we must have
$$
\nabla-({}^g\nabla+g \otimes X_g)=\iota(\beta),
$$
for some $1$-form $\beta$ on $M$. Adding $\iota((X_g)^{\flat})$ gives
$$
\nabla-\left({}^g\nabla+g\otimes X_g-\iota\left((X_g)^{\flat}\right)\right)=\iota\left(\beta+(X_g)^{\flat}\right),
$$
so that Lemma~\ref{Weylfund} implies that $\nabla$ and the $[g]$-conformal connection
$$
{}^g\nabla+g\otimes X_g-\iota\left((X_g)^{\flat}\right)
$$
are projectively equivalent. 

(vi) Let now $n=2$. We need to show that for $g \in [g]$ and all vector fields $X,Y,Z \in \Gamma(TM)$, we have
$$
g(A_{[g]}(X)Y,Z)=g(Y,A_{[g]}(X)Z). 
$$
Without loosing generality, we can assume that locally $g=(\d x^1)^2+(\d x^2)^2$ for coordinates $x=(x^1,x^2) : U \to \R^2$ on $M$. Let $\Gamma^i_{jk}$ denote the Christoffel symbols of $\nabla$ with respect to $x$. Since the Christoffel symbols of ${}^g\nabla$ vanish identically on $U$, we obtain with a simple calculation 
$$
X_g=-\frac{3}{4}\left(w_1+w_3\right)\frac{\partial}{\partial x^1}+\frac{3}{4}\left(w_0+w_2\right)\frac{\partial}{\partial x^2},
$$
where
$$
w_0=\Gamma^2_{11}, \quad 3w_1=-\Gamma^1_{11}+2\Gamma^2_{12}, \quad 3w_2=-2\Gamma^1_{12}+\Gamma^2_{22}, \quad w_3=-\Gamma^1_{22}.
$$
Likewise, we compute
\begin{multline*}
A_{[g]}=\frac{1}{2}\left(a_1e^{11}_{\phantom{11}1}-a_2e^{11}_{\phantom{11}2}-a_2e^{12}_{\phantom{12}1}-a_1e^{12}_{\phantom{12}2}\right.\\-\left.a_2e^{21}_{\phantom{21}1}-a_1e^{21}_{\phantom{21}2}-a_1e^{22}_{\phantom{22}1}+a_2e^{22}_{\phantom{22}2}\right)
\end{multline*}
where we write $e^{ij}_{\phantom{ij}k}=\d x^i\otimes \d x^j\otimes \frac{\partial}{\partial x^k}$
and
$$
a_1=\frac{1}{2}(w_3-3w_1), \quad a_2=\frac{1}{2}(3w_2-w_0). 
$$
The claim follows from an elementary calculation. 
\end{proof}
\begin{rmk}
By construction, the $1$-form $A_{[g]}$ vanishes identically if and only if $\nabla$ is projectively equivalent to a conformal connection. The necessary and sufficient conditions for a torsion-free connection to be projectively equivalent to a Levi-Civita connection were given in~\cite{MR2581355}. The reader may also consult~\cite{MR3158041} for the role of Einstein metrics in projective differential geometry. 
\end{rmk}
As a corollary to Theorem~\ref{quasilinop1} and Lemma~\ref{intersect} we obtain the following result.
\begin{cor}\label{uniquepoints}
For every conformal structure $[g]$ on the projective mani\-fold $(M,\mathfrak{p})$, there exists a unique $[g]$-conformal connection $\con$ so that $\con+A_{[g]}\in \mathfrak{p}$.  
\end{cor}
Note that Corollary~\ref{uniquepoints} provides a unique point $\con \in \mathfrak{A}_{[g]}(M)$ and a unique point $\con+A_{[g]} \in \mathfrak{A}_{\mathfrak{p}}(M)$. We may define
$$
\psi\left(\mathfrak{p},[g]\right)=\left(\con+A_{[g]},\con\right).
$$
%where $\con$ is the unique conformal connection provided by Corollary~\ref{uniquepoints}. 
Since the map which sends a Riemannian metric to its Levi-Civita connection is equivariant with respect to the action of $\mathrm{Diff}(M)$ on the space of Riemannian metrics and on $\mathfrak{A}(M)$, it follows that the map $\psi$ has all the properties listed in Remark~\ref{rmk:diffequimap}.    
\begin{proof}[Proof of Corollary~\ref{uniquepoints}]
Let $\nabla$ be a connection defining $\mathfrak{p}$ and $g$ a smooth metric defining $[g]$. Set
$$
\con={}^g\nabla+g\otimes X_g-(X_g)^{\flat}\otimes \mathrm{Id}-\mathrm{Id}\otimes (X_g)^{\flat},
$$
where $X_g$ is defined as before (see~\eqref{defvect}). Then, property (i) of $A_{[g]}$ proved in Theorem~\ref{quasilinop1} implies that
$$
\left(\nabla-\left({}^{[g]}\nabla+A_{[g]}\right)\right)_0=\left(\nabla-({}^g\nabla+g\otimes X_g)\right)_0-A_{[g]}=A_{[g]}-A_{[g]}=0,
$$ 
so that $\con+A_{[g]}$ is projectively equivalent to $\nabla$ by Lemma~\ref{Weylfund}. If $\con^{\prime}$ is another $[g]$-conformal connection so that $\con^{\prime}+A_{[g]}$ defines $\mathfrak{p}$, then 
$$
\left(\con-\con^{\prime}\right)_0=0,
$$
hence $\con=\con^{\prime}$ by Lemma~\ref{intersect}. 
\end{proof}

\subsection{A diffeomorphism invariant functional}

We will henceforth assume $M$ to be oriented. For a pair $(\mathfrak{p},[g])$ consisting of a projective structure and a conformal structure on $M$, we consider the non-negative $n$-form $|A_{[g]}|^n_gd\mu_g$, where $g$ is any metric defining $[g]$, the $n$-form $d\mu_g$ denotes its volume form and where $A_{[g]}$ is computed with respect to $\mathfrak{p}$. For $f \in C^{\infty}(M)$ we have 
$$
|A_{[g]}|_{\e^{2f}g}=\e^{-f}|A_{[g]}|_{g} \quad \text{and} \quad d\mu_{\e^{2f}g}=e^{nf}d\mu_g,
$$
it follows that $|A_{[g]}|^n_gd\mu_g$ depends only on the conformal structure $[g]$.
Consequently, we obtain a non-negative functional
$$
\mathcal{F} : \mathfrak{P}(M)\times \mathfrak{C}(M) \to \R^{+}_{0}\cup \{\infty\}, \quad (\mathfrak{p},[g]) \mapsto \int_{M} |A_{[g]}|^n_gd\mu_g.
$$
By construction, $\mathcal{F}$ is invariant under simultaneous action of $\mathrm{Diff}(M)$ on $\mathfrak{P}(M)$ and $\mathfrak{C}(M)$.  

We may also fix a projective structure $\mathfrak{p}$ on $M$ and define $\mathcal{E}_{\mathfrak{p}}=\mathcal{F}[(\mathfrak{p},\cdot)]$ which is a functional on $\mathfrak{C}(M)$ only. We may study the infimum of $\mathcal{E}_{\mathfrak{p}}$ among all conformal structures on $M$, and ask whether there is actually a minimising conformal structure which achieves this infimum. The infimum
$$
\Gamma\delta(M,\mathfrak{p}):=\inf_{[g] \in \mathfrak{C}(M)}\mathcal{E}_{\mathfrak{p}}([g]),
$$ 
which may be considered as a measure of how far $\mathfrak{p}$ deviates from being defined by a conformal connection, is a new global invariant for oriented projective manifolds. Note that reversing the role of $\mathfrak{p}$ and $[g]$ does not give us a global invariant for conformal manifolds. Clearly, fixing a conformal structure and considering the infimum over $\mathfrak{P}(M)$ yields zero for every choice of conformal structure $[g]$.

\section{Projective surfaces and associated bundles}

A natural case to consider is $n=2$, where $\mathcal{F}$ is just the square of the $L^2$-norm of $A_{[g]}$ taken with respect to $[g]$. We will henceforth consider the surface case only. 

There are several natural geometric spaces fibering over an oriented projective surface which we will discuss next. Before doing so, we recall a result of Cartan~\cite{MR1504846}, which canonically associates a principal bundle together with a ``connection'' to every projective manifold. The reader interested in a description of Cartan's construction using modern language may also consult~\cite{MR0159284}. For additional background on Cartan geometries the reader may also consult~\cite{MR2532439}.  

\subsection{Cartan's normal projective connection}\label{sec:cartbund}

Let $\Sigma$ be an oriented surface and let $\mathrm{G}\simeq \R_2\rtimes \mathrm{GL}^+(2,\R)$ denote the two-dimensional orientation preserving affine group which we think of as the subgroup of $\mathrm{SL}(3,\R)$ consisting of matrices of the form
$$
b\rtimes a=\begin{pmatrix} \det a^{-1} & b \\ 0 & a\end{pmatrix},
$$
for $b \in \R_2$ and $a \in \mathrm{GL}^+(2,\R)$. 
We denote by $\upsilon : F^+\to \Sigma$ the principal right $\mathrm{GL}^+(2,\R)$-bundle of coframes that are orientation preserving with respect to the chosen orientation on $\Sigma$ and the standard orientation on $\R^2$. We define a right $\mathrm{G}$-action on $F^+\times \R_2$ by the rule
\begin{equation}\label{rightactioncartbundleisom}
(u,\xi)\cdot (b\rtimes a)=\left(\det a^{-1} a^{-1} \circ u,\xi a \det a + b\det a\right),
\end{equation}
for all $b\rtimes a \in \mathrm{G}$. Here $\xi : F^+ \times \R_2 \to \R_2$ denotes the projection onto the latter factor. This action turns $\pi : F^+\times \R_2 \to \Sigma$ into a principal right $\mathrm{G}$-bundle over $\Sigma$, where $\pi : F^+\times \R_2\to \Sigma$ denotes the natural basepoint projection. Suppose $\nabla$ is a torsion-free connection on $T\Sigma$ with connection $1$-form $\eta=(\eta^i_j)$ on $F^+$ so that we have the structure equations\footnote{Indices in round brackets are symmetrised over and indices in square brackets are anti-symmetrised over, for instance, we write $S_{(ij)}=\frac{1}{2}\left(S_{ij}+S_{ji}\right)$ and $S_{[ij]}=\frac{1}{2}\left(S_{ij}-S_{ji}\right)$ so that $S_{ij}=S_{(ij)}+S_{[ij]}$.}
\begin{align*}
\d \omega^i&=-\eta^i_j\wedge\omega^j,\\
\d\eta^i_j&=-\eta^i_k\wedge\eta^k_j+(\delta^i_{[k}S_{l]j}-S_{[kl]}\delta^i_j)\omega^k\wedge\omega^l,
\end{align*}
where $S=(S_{ij})$ represents the \textit{projective Schouten tensor} $\mathrm{Schout}(\nabla)$ of $\nabla$ and $\omega^i$ the components of the tautological $\R^n$-valued $1$-form $\omega$ on $F$. Recall that the Schouten tensor is defined as 
\begin{equation}\label{defschout}
\mathrm{Schout}(\nabla)=\mathrm{Ric}^{+}(\nabla)-\frac{1}{3}\mathrm{Ric}^{-}(\nabla),
\end{equation}
where $\mathrm{Ric}^{\pm}(\nabla)$ denote the symmetric and anti-symmetric part of the Ricci curvature of $\nabla$. On $\projb=F^+\times \R_2$ we define the $\mathfrak{sl}(3,\R)$-valued $1$-form
\begin{equation}\label{cartconpref}
\theta=\begin{pmatrix} -\frac{1}{3}\tr \eta-\xi \omega & \d \xi-\xi\eta-(S\omega)^t-\xi\omega\xi\\  \omega & \eta-\frac{1}{3}\mathrm{I}\tr \eta+\omega\xi  \end{pmatrix}.
\end{equation}
The reader may check that the pair $(\pi : \projb \to \Sigma,\theta)$ defines a~\textit{Cartan geometry} of type $(\mathrm{SL}(3,\R),\mathrm{G})$, that is, $\pi : \projb \to \Sigma$ is a principal right $\mathrm{G}$-bundle and $\theta$ is an $\mathfrak{sl}(3,\R)$-valued $1$-form on $\projb$ satisfying the following properties:
\begin{itemize}
\item[(i)] $\theta_u : T_u\projb \to \mathfrak{sl}(3,\R)$ is an isomorphism for every $u \in \projb$; 
\item[(ii)] $(R_{g})^*\theta=g^{-1}\theta g$ for every $g \in \mathrm{G}$; 
\item[(iii)] $\theta(X_v)=v$ for every fundamental vector field $X_v$ generated by an element $v$ in the Lie algebra of $\mathrm{G}$. 
\end{itemize}  
Moreover, writing $\theta=(\theta^i_j)_{i,j=0,1,2}$, the Cartan geometry $(\pi : \projb \to \Sigma,\theta)$ also satisfies:
\begin{itemize}
\item[(iv)] for every non-zero $x \in \R^2$, the integral curves of the vector field $X_{x}$ defined by the equations 
$$
\theta(X_x)=\begin{pmatrix} 0 & 0 \\ x & 0 \end{pmatrix}
$$
project to $\Sigma$ to become geodesics of $\mathfrak{p}$ and conversely all geodesics of $\mathfrak{p}$ arise in this way; 
\item[(v)] the $\pi$-pullback of an orientation compatible volume form on $\Sigma$ is a positive multiple of $\theta^1_0\wedge\theta^2_0$;
\item[(vi)] the curvature $2$-form $\Theta=\d\theta+\theta\wedge\theta$ is
\begin{equation}\label{realstruceqproj}
\Theta=\d\theta+\theta\wedge\theta=\begin{pmatrix} 0 & L_1\theta^1_0\wedge\theta^2_0 & L_2 \theta^1_0\wedge\theta^2_0 \\ 0 & 0 & 0 \\ 0 & 0 & 0\end{pmatrix}, 
\end{equation}
for unique curvature functions $L_i : \projb \to \R$. 
\end{itemize} 
\begin{rmk}
Cartan's bundle is unique in the following sense: If $(\hat{\pi} : \hat{\projb} \to \Sigma,\hat{\theta})$ is another Cartan geometry of type $\left(\mathrm{SL}(3,\R),\mathrm{G}\right)$ so that the properties (iv),(v) and (vi) hold, then there exists a $\mathrm{G}$-bundle isomorphism $\psi : \projb \to \hat{\projb}$ satisfying $\psi^*\hat{\theta}=\theta$.  
\end{rmk}
A projective structure $\mathfrak{p}$ on $\Sigma$ is called~\textit{flat} if every point $p \in \Sigma$ has a neighbourhood $U_p$ which is diffeomorphic to a subset of $\mathbb{RP}^2$ in such a way that the geodesics of $\mathfrak{p}$ contained in $U_p$ are mapped onto (segments) of projective lines $\mathbb{RP}^1\subset \mathbb{RP}^2$. Furthermore, a torsion-free connection $\nabla$ on $T\Sigma$ is called~\textit{projectively flat} if $\mathfrak{p}(\nabla)$ is flat. Using Cartan's connection, one can show that a projective structure $\mathfrak{p}$ is flat if and only if the functions $L_1$ and $L_2$ vanish identically. Another consequence of Cartan's result is that there exists a unique $1$-form $\lambda\in\Omega^1(\Sigma,\Lambda^2(T^*\Sigma))$ so that 
$$
\pi^*\lambda=(L_1\theta^1_0+L_2\theta^2_0)\otimes \theta^1_0\wedge\theta^2_0. 
$$
The $1$-form $\lambda$ was first discovered by R.~Liouville~\cite{19.0317.02}, hence we call $\lambda$ the~\textit{Liouville curvature} of $\mathfrak{p}$. In particular, the Liouville curvature is the complete obstruction to flatness of a two-dimensional projective structure.

\begin{ex}
Note that the left action of $\mathrm{SL}(3,\R)$ on $\R^3$ by matrix multiplication descends to define a transitive left action on the projective $2$-sphere $\mathbb{S}^2$. The stabiliser subgroup of the element $[(1\;0\;0)^t]$ is the group $\mathrm{G}\subset \mathrm{SL}(3,\R)$ so that $\mathbb{S}^2\simeq \mathrm{SL}(3,\R)/\mathrm{G}$. Taking $\theta$ to be the Maurer-Cartan form of $\mathrm{SL}(3,\R)$, the pair $(\pi : \mathrm{SL}(3,\R)\to \mathbb{S}^2,\theta)$ is a Cartan geometry of type $(\mathrm{SL}(3,\R),\mathrm{G})$ defining an orientation and projective structure $\mathfrak{p}_{\mathrm{can}}$ on $\mathbb{S}^2$ whose geodesics are the ``great circles''. Since $\d\theta+\theta\wedge\theta=0$, this projective structure is flat. We call $\mathfrak{p}_{\mathrm{can}}$ the~\textit{canonical flat projective structure} on $\mathbb{S}^2$. 
\end{ex}

\subsection{The twistor space}

Inspired by Hitchin's twistorial description of holomorphic projective structures on complex surfaces~\cite{MR699802}, it was shown in~\cite{MR728412,MR812312} how to construct a ``twistor space'' for smooth projective structures. For what follows it will be convenient to construct the twistor space in the smooth category by using the Cartan geometry of a projective surface. 

Let therefore $(\Sigma,\mathfrak{p})$ be an oriented projective surface with Cartan geometry $(\pi : \projb \to \Sigma,\theta)$. By construction, the quotient of $\projb$ by the normal subgroup $\R_2\rtimes\mathrm{\{Id\}}\subset \mathrm{G}$ is isomorphic to the bundle $\upsilon : F^+ \to \Sigma$ of orientation preserving coframes of $\Sigma$. In particular, the choice of a conformal structure $[g]$ on $\Sigma$ corresponds to a section of the fibre bundle $\mathrm{C}(\Sigma)\simeq \projb/\left(\R_2\rtimes\mathrm{CO}(2)\right) \to \Sigma$. Here $\mathrm{CO}(2)=\R^+\times\mathrm{SO}(2)$ is the linear orientation preserving conformal group. By construction, the typical fibre of the bundle $\mathrm{C}(\Sigma) \to \Sigma$ is diffeomorphic to $\mathrm{GL}^+(2,\R)/\mathrm{CO}(2)\simeq\mathrm{SL}(2,\R)/\mathrm{SO}(2)$, that is, the open unit disk $D^2\subset \C$.  

We write the elements of the group $\R_2\rtimes \mathrm{CO}(2)$ in the following form
$$
z\rtimes r\mathrm{e}^{\i \phi}=\begin{pmatrix} r^{-2} & \Re(z) & \Im(z)\\ 0& r\cos\phi & r \sin\phi \\ 0 & -r \sin\phi & r\cos\phi\end{pmatrix}, \quad z \in \C,\rot \in \C^*. 
$$
Property (iii) of the Cartan geometry $(\pi : \projb\to \Sigma,\theta)$ implies that the (real -- or complex-valued) $1$-forms on $\projb$ that are semibasic\footnote{Recall that a differential form $\alpha$ is said to be semibasic for the projection $\projb \to \mathrm{C}(\Sigma)$ if the interior product $X\inc \alpha$ vanishes for every vector field $X$ tangent to the fibres of $\projb \to \mathrm{C}(\Sigma)$.} for the quotient projection $\mu : \projb \to \mathrm{C}(\Sigma)$ are complex linear combinations of the complex-valued $1$-forms 
\begin{equation}\label{defformsZ}
\zeta_1=\theta^1_0+\i \theta^2_0, \qquad \zeta_2=\left(\theta^1_1-\theta^2_2\right)+\i\left(\theta^1_2+\theta^2_1\right)
\end{equation}
and their complex conjugates. The equivariance property (ii) of the Cartan geometry gives\begin{equation}\label{rightacttwist}
\left(R_{z\rtimes\rot}\right)^*\begin{pmatrix}\zeta_1 \\ \zeta_2\end{pmatrix}=\begin{pmatrix} \frac{1}{r^3} \e^{\i\phi} & 0\\ \frac{z}{r}\e^{\i\phi} & \e^{2\i\phi} \end{pmatrix}\begin{pmatrix}\zeta_1 \\ \zeta_2\end{pmatrix}.
\end{equation}
It follows that there exists a unique almost complex structure $\mathfrak{J}$ on $\mathrm{C}(\Sigma)$ having the property that a complex-valued $1$-form on $\projb$ is the pullback of a $(1,\! 0)$-form on $\mathrm{C}(\Sigma)$ if and only if it is a complex linear combination of $\zeta_1$ and $\zeta_2$. Indeed, locally we may use a section $s$ of the bundle $\mu : \projb \to \mathrm{C}(\Sigma)$ to pull down the forms $\zeta_1,\zeta_2$ onto the domain of definition $U\subset \mathrm{C}(\Sigma)$ of $s$. Since $\zeta_1,\zeta_2$ are semi-basic for the projection $\mu : \projb \to \mathrm{C}(\Sigma)$, it follows that the pulled down forms are linearly independent over $\C$ at each point of $U$. Hence we obtain a unique almost complex structure $\mathfrak{J}$ on $U$ whose $(1,\! 0)$-forms are $s^*\zeta_1,s^*\zeta_2$. The equivariance~\eqref{rightacttwist} implies that $\mathfrak{J}$ is independent of the choice of the section $s$ and extends to all of $\mathrm{C}(\Sigma)$. Using property (vi) of the Cartan geometry the reader may easily verify that
$$
\d\zeta_1=\d\zeta_2=0, \quad \text{mod}\quad \zeta_1,\zeta_2.  
$$
It follows from the Newlander-Nirenberg theorem that $\mathfrak{J}$ is integrable, thus giving $\mathrm{C}(\Sigma)$ the structure of a complex surface which we will denote by $Z$ and which -- abusing language -- we call the~\textit{twistor space} of the projective surface $(\Sigma,\mathfrak{p})$. 

\subsection{An indefinite K\"ahler-Einstein 3-fold}

From~\eqref{rightacttwist} it follows that the holomorphic cotangent bundle $T^*_{\C}Z^{1,0}\to Z$ is the bundle associated to $\mu : \projb \to Z$ via the complex two-dimensional representation $\rho : \R_2\rtimes \mathrm{CO}(2) \to \mathrm{GL}(2,\C)$ defined by the rule
\begin{equation}\label{firstrep}
\rho(z \rtimes \rot)(w_1\;w_2)=(w_1\;w_2)\begin{pmatrix} \frac{1}{r^3} \e^{\i\phi} & 0\\ \frac{z}{r}\e^{\i\phi} & \e^{2\i\phi} \end{pmatrix}
\end{equation}
for all $(w_1\;w_2) \in \C_2$. In particular, the form $\zeta_1$ is well defined on $Z$ up to complex-scale and hence may be thought of as a section of the projective holomorphic cotangent bundle $\mathbb{P}(T^*_{\C}Z^{1,0}) \to Z$. Abusing notation, we write $\zeta_1(Z) \subset \mathbb{P}(T^*_{\C}Z^{1,0})$ to denote the image of $Z$ under this section.
We now have:
\begin{lem}\label{cplx3fold}
There exists a unique integrable almost complex structure on the quotient $\projb/\mathrm{CO}(2)$ having the property that its $(1,\! 0)$-forms pull back to $\projb$ to become linear combinations of the forms
\begin{equation}\label{eq:defkaehforms}
\zeta_1=\theta^1_0+\i\theta^2_0,\quad \zeta_2=\left(\theta^1_1-\theta^2_2\right)+\i\left(\theta^1_2+\theta^2_1\right),\quad \zeta_3=\theta^0_1+\i\theta^0_2.
\end{equation}
Furthermore, with respect to this complex structure $\projb/\mathrm{CO}(2)$ is biholomorphic to $Y=\mathbb{P}(T^*_{\C}Z^{1,0})\setminus \zeta_1(Z)$ in such a way that the standard holomorphic contact structure on $Y$ is identified with the subbundle of $T_{\C}(\projb/\mathrm{CO}(2))^{1,0}$ defined by the equation $\zeta_2=0$. 
\end{lem}
\begin{proof}
Again, it follows from the property (iii) of the Cartan connection $\theta$ that the $1$-forms that are semibasic for the quotient projection $\tau :\projb \to \projb/\mathrm{CO}(2)$ are linear combinations of the forms $\zeta_1,\zeta_2,\zeta_3$ and their complex conjugates. Here $\mathrm{CO}(2)\subset \mathrm{G}$ is the subgroup consisting of elements of the form $0\rtimes \rot$. Writing $\rot$ instead of $0\rtimes \rot$ and $\zeta=(\zeta_i)$, we compute from the equivariance property (ii) of $\theta$ that we have 
\begin{equation}\label{righttranskae}
\left(R_{r\e^{\i\phi}}\right)^*\begin{pmatrix}\zeta_1\\ \zeta_2\\ \zeta_3\end{pmatrix}=\begin{pmatrix}\frac{1}{r^3}\e^{\i\phi} & 0 & 0\\ 0 & \e^{2\i\phi} & 0 \\ 0 & 0 & r^3\e^{\i\phi}\end{pmatrix}\begin{pmatrix}\zeta_1\\ \zeta_2\\ \zeta_3\end{pmatrix}.
\end{equation}
It follows as before that there exists a unique almost complex structure $\mathfrak{J}$ on $\projb/\mathrm{CO}(2)$ having the property that its $(1,\! 0)$-forms pull back to $\projb$ to become linear combinations of the forms $\zeta_1,\zeta_2,\zeta_3$. Suppose there exists a $1$-form $\gamma=(\gamma_{ij})$ on $P$ with values in $\mathfrak{gl}(3,\C)$, so that $\d\zeta=-\gamma\wedge\zeta$, then it follows again from the Newlander--Nirenberg theorem that $\mathfrak{J}$ is integrable. Clearly, if such a $\gamma$ exists, then it is not unique. Defining $\hat{\gamma}=(\hat{\gamma}_{ij})$, with $\hat{\gamma}_{ij}=\gamma_{ij}+T_{ijk}\zeta_k$ for some complex-valued functions satisfying $T_{ijk}=T_{ikj}$ on $P$ will also work. We can exploit this freedom and make $\gamma$ take values in the Lie algebra 
$$
\mathfrak{u}(2,1)=\left\{\begin{pmatrix} w_1 & -\ov{w_2} & \i x_1\\ -w_3 & \i x_2 & w_2 \\ \i x_3 & \ov{w_3} & -\ov{w_1} \end{pmatrix}\,: \, w_1,w_2,w_3 \in \C\;\text{and}\; x_1,x_2,x_3 \in \R\right\}
$$ of the indefinite unitary group $\mathrm{U}(2,\! 1)$, where the model of $\mathrm{U}(2,\! 1)$ being used is the subgroup of $\mathrm{GL}(3,\C)$ that fixes the Hermitian form in $3$-variables
$$
H(z)=z_1\ov{z_3}+z_3\ov{z_1}+z_2\ov{z_2}. 
$$
Indeed, writing 
\begin{equation}\label{eq:defcurvandcon}
L=-\frac{1}{2}\left(L_2-\i L_1\right)\quad\text{and}\quad \varphi=-\frac{1}{2}\left(3\theta^0_0+\i(\theta^1_2-\theta^2_1)\right),
\end{equation}
we have
\begin{equation}\label{righttranskaeinf}
\d \zeta=-\gamma\wedge\zeta,
\end{equation}
where
$$
\gamma=\begin{pmatrix} \varphi & -\frac{1}{2}\ov{\zeta_1} & 0 \\-\frac{1}{2}\zeta_3 & \varphi-\ov{\varphi} & \frac{1}{2}\zeta_1\\ L\ov{\zeta_1}-\ov{L}\zeta_1 &\frac{1}{2}\ov{\zeta_3} &-\ov{\varphi}\end{pmatrix}, 
$$
as the reader can verify by using the definitions~\eqref{eq:defkaehforms},\eqref{eq:defcurvandcon} and the structure equations~\eqref{realstruceqproj}. It follows that $\mathfrak{J}$ is integrable. Likewise, the reader may verify that
\begin{equation}\label{struceqpsi}
\d \varphi=\frac{1}{2}\zeta_3\wedge\ov{\zeta_1}-\frac{1}{4}\zeta_2\wedge\ov{\zeta_2}-\zeta_1\wedge\ov{\zeta_3},
\end{equation}
simply by plugging in the definitions of the involved forms and by using the structure equations~\eqref{realstruceqproj}. 

Now consider the map 
$$
\tilde{\psi} : \projb \to \projb \times \C_2\setminus\{0\}, \quad u \mapsto \left(u,\begin{pmatrix} 0 & 1\end{pmatrix}\right)  
$$
and let $q : \projb \times \C_2\setminus\{0\} \to \mathbb{P}(T^*_{\C}Z^{1,0})$ denote the natural quotient projection induced by (the projectivisation of) $\rho$. Then $q \circ \tilde{\psi} : \projb \to \mathbb{P}(T^*_{\C}Z^{1,0})$ is a submersion onto $Y$ whose fibres are the $\mathrm{CO}(2)$-orbits. Indeed, let $(u,w)$ be a representative of an element $[\nu]\in\mathbb{P}(T^*_{\C}Z^{1,0})$ which lies in the complement of $\zeta_1(Z)$. Then using~\eqref{firstrep} it follows that we might transform with the affine part of the right action of $\R_2\rtimes \mathrm{CO}(2)$ to ensure that $w$ is of the form $(0\; w_2)$ for some non-zero complex number $w_2$. It follows that the element $u \in \projb$ is mapped onto $[\nu]$ showing that $q \circ \tilde{\psi}$ is surjective onto $Y$. Clearly $q \circ \tilde{\psi}$ is smooth and a submersion. Furthermore, suppose the two points $u,u^{\prime} \in \projb$ are mapped to the same element of $Y$. Then, there exists an element $z\rtimes \rot \in \R_2\rtimes \mathrm{CO}(2)$ and a non-zero complex number $s$ so that
$$
\rho\left((z\rtimes \rot)^{-1}\right)\begin{pmatrix}0 & 1\end{pmatrix}=\begin{pmatrix}-zr^2\e^{-2\i\phi} & \e^{-2\i\phi}\end{pmatrix}=\begin{pmatrix} 0 & s\end{pmatrix} 
$$
which holds true if and only if $z=0$. Consequently, there exists a unique diffeomorphism $\psi : \projb/\mathrm{CO}(2) \to Y$ making the following diagram commute:
$$
\begin{xy}
(0,0)*+{\projb}="\projb";
(30,0)*+{\projb\times \C_2\setminus\{0\}}="up";
(0,-15)*+{\projb/\mathrm{CO}(2)}="down";
(30,-15)*+{Y}="Y";
{\ar "\projb";"up"}?*!/_3mm/{\tilde{\psi}};
{\ar "\projb";"down"}?*!/^3mm/{\tau};
{\ar "down";"Y"}?*!/_3mm/{\psi};
{\ar "up";"Y"}?*!/_3mm/{q}
\end{xy}
$$
The complex structure on $Y \subset \mathbb{P}(T^*_{\C}Z^{1,0})$ is such that its $(1,\! 0)$-forms pull back to $\projb\times \C_2\setminus\{0\}$ to become linear combinations of the complex-valued $1$-forms $\zeta_1,\zeta_2,\d w_1, \d w_2$, where $w=(w_1 \; w_2) : \projb\times \mathbb{C}_2 \to \C_2$ denotes the projection onto the linear factor. Clearly, these forms pull back under $\tilde{\psi}$ to become linear combinations of the forms $\zeta_1,\zeta_2,\zeta_3$, hence $\psi$ is a biholomorphism.

Finally, note that the complex version of the Liouville $1$-form on $T^*_{\C}Z^{1,0}$ -- whose kernel defines the canonical contact structure on $\mathbb{P}(T^{*}_{\C}Z^{1,0})$ -- pulls back to $\projb\times \mathbb{C}_2$ to become $w_1\zeta_1+w_2\zeta_2$. Since
$$
\tilde{\psi}^*\left(w_1\zeta_1+w_2\zeta_2\right)=\zeta_2,
$$
the claim follows. 
\end{proof}
\begin{rmk}
Whereas the definition of the forms $\zeta_i$ is a natural consequence of the Lie algebra structure of $\mathrm{CO}(2)\subset \R_2\rtimes \mathrm{GL}^+(2,\R)$, the definition of the form $\varphi$ in~\eqref{eq:defcurvandcon} is somewhat mysterious at this point. The choice will be clarified during the proof of Proposition~\ref{ppn:rerho=alpha} below.  
\end{rmk}
%\begin{rmk}
%Alternatively, it follows from~\eqref{righttranskae} that the equations $\zeta_2=0$ define a subbundle $C$ of $T_{\C}(\projb/\mathrm{CO}(2))^{1,0}\simeq T_{\C}Y^{1,0}$ and furthermore~\eqref{righttranskaeinf} yields 
%$$
%\d\zeta_2\wedge\zeta_2=\zeta_1\wedge\zeta_2\wedge\zeta_3\neq 0,
%$$
%confirming that $C$ is a holomorphic contact structure.  
%\end{rmk}
We will henceforth identify $Y\simeq \projb/\mathrm{CO}(2)$ and think of $\tau$ as the projection map onto $Y$. Denoting the integrable almost complex structure on $Y$ by $J$, the first part of the following proposition is therefore clear:
\begin{ppn}
There exists a unique indefinite K\"ahler structure on $(Y,J)$ whose K\"ahler-form $\Omega_Y$ satisfies
\begin{align*}
\tau^*\Omega_Y&=-\frac{\i}{4}\left(\zeta_1\wedge \ov{\zeta_3}+\zeta_3\wedge\ov{\zeta_1}+\zeta_2\wedge\ov{\zeta_2}\right).
\end{align*}
Moreover, the indefinite K\"ahler metric $\Ym(\cdot,\cdot):=\Omega_Y(J\cdot,\cdot)$ is Einstein with non-zero scalar curvature. 
\end{ppn}
\begin{proof}
The first part of the statement is an immediate consequence of the fact that $\gamma$ takes values in $\mathfrak{u}(2,1)$. The skeptical reader might also verify this using the structure equations~\eqref{righttranskaeinf}. Furthermore, by definition, the associated K\"ahler metric satisfies
$$
\tau^*h=\frac{1}{2}\left(\zeta_1\circ \ov{\zeta_3}+\zeta_3\circ\ov{\zeta_1}+\zeta_2\circ\ov{\zeta_2}\right)\\
$$   
and hence the forms $\frac{1}{\sqrt{2}}\zeta_i$ are a unitary coframe for $\tau^*\Ym$. In order to verify the Einstein condition it is therefore sufficient that the trace of the curvature form
$$
\Gamma=\d \gamma+\gamma\wedge\gamma
$$
is a non-zero constant (imaginary) multiple of $\tau^*\Omega_Y$. We compute
\begin{align*}
0&=\d^2\zeta_3\wedge\zeta_1\wedge\ov{\zeta_1}=-\d\left(\gamma_{31}\wedge\zeta_1+\gamma_{32}\wedge\zeta_2+\gamma_{33}\wedge\zeta_3\right)\wedge\zeta_1\wedge\ov{\zeta_1}\\
&=\zeta_1\wedge\ov{\zeta_1}\wedge\left(\d L+\frac{1}{2}\ov{L}\zeta_2-L\varphi-2L\ov{\varphi}\right),
\end{align*} 
where we have used~\eqref{righttranskaeinf} and~\eqref{struceqpsi}. It follows that there exist unique complex-valued functions $L^{\prime}$ and $L^{\prime\prime}$ on $P$ such that
\begin{equation}\label{eq:struceqliouv}
\d L=L^{\prime}\zeta_1+L^{\prime\prime}\ov{\zeta_1}-\frac{1}{2}\ov{L}\zeta_2+L\varphi+2L\ov{\varphi}. 
\end{equation}
Using the structure equations~\eqref{righttranskaeinf},\eqref{struceqpsi} and~\eqref{eq:struceqliouv} we compute
$$
\Gamma=\frac{1}{4}\begin{pmatrix} \Gamma_{11} & -\zeta_1\wedge\ov{\zeta_2} & \zeta_1\wedge\ov{\zeta_1}\\ -\zeta_2\wedge\ov{\zeta_3} & \Gamma_{22} & \ov{\zeta_1}\wedge\zeta_2\\ \zeta_3\wedge\ov{\zeta_3}+* & \ov{\zeta_2}\wedge\zeta_3& \Gamma_{33}\end{pmatrix},
$$
with
\begin{align*}
\Gamma_{11}&=\frac{1}{4}\left(\zeta_3\wedge\ov{\zeta_1}-\zeta_2\wedge\ov{\zeta_2}-4\zeta_1\wedge\ov{\zeta_3}\right),\\
\Gamma_{22}&=\frac{1}{4}\left(-\zeta_1\wedge\ov{\zeta_3}-2\zeta_2\wedge\ov{\zeta_2}-\zeta_3\wedge\ov{\zeta_1}\right),\\
\Gamma_{33}&=\frac{1}{4}\left(\zeta_1\wedge\ov{\zeta_3}-\zeta_2\wedge\ov{\zeta_2}-4\zeta_3\wedge\ov{\zeta_1}\right).
\end{align*}
and where 
$
*=4\left(L^{\prime}+\ov{L^{\prime}}\right)\zeta_1\wedge\ov{\zeta_1}
$. In particular, we obtain
$$
\Gamma_{11}+\Gamma_{22}+\Gamma_{33}=4\i \tau^*\Omega_Y, 
$$
thus verifying the Einstein property. 
\end{proof}
\begin{rmk}
In~\cite{MR3833818}, it is shown how to canonically associate a split-signature anti-self-dual Einstein metric on the total space of a certain rank two affine bundle $A$ fibering over a projective surface $(\Sigma,\mathfrak{p})$. The indefinite K\"ahler--Einstein manifold $(Y,J,\Omega_Y$) constructed here may be interpreted as the twistor space of this anti-self-dual Einstein metric. 
\end{rmk}
\subsection{The canonical flat case}
In this subsection we identify the spaces 
$$
Y=\projb/\mathrm{CO}(2) \quad\text{and} \quad Z=\projb/\left(\R_2\rtimes \mathrm{CO}(2)\right)
$$
in the case where $(\Sigma,\mathfrak{p})$ is the canonical flat projective structure on the projective $2$-sphere. Recall that in this case $\projb=\mathrm{SL}(3,\R)$. The group $\mathrm{SL}(3,\R)$ also acts naturally on $\C_3$ by complexification, that is, by the rule
$$
g \cdot (\xi+\i \chi)=\xi g^{-1}+\i \chi g^{-1}
$$
for all $g \in \mathrm{SL}(3,\R)$. Clearly, this action descends to define a left action on $\mathbb{CP}_2$. However, this action is not transitive, but has two orbits. The first orbit is $\mathbb{RP}_2\subset \mathbb{CP}_2$, where we think of $\mathbb{RP}_2$ as those points $[\xi+\i \chi] \in \mathbb{CP}_2$ which satisfy $\xi\wedge\chi=0$, that is, $\xi$ and $\chi$ are linearly dependent over $\R$. Assume therefore $[\eps]$ is an element in the complement $\mathbb{CP}_2\setminus\mathbb{RP}_2$ of $\mathbb{RP}^2$ in $\mathbb{CP}^2$. Since $\mathrm{SL}(3,\R)$ acts transitively on unimodular triples of vectors in $\R_3$, we can assume without losing generality that $\eps=(0\;-\i \; 1)$. For $g\in\mathrm{SL}(3,\R)$ we write $g=(g_0\;g_1\;g_2)$ with $g_i \in \R^3$. We will next determine the stabiliser subgroup of $[\eps]$. A simple computation gives
$$
g\cdot \eps=g_0\wedge\left(g_1+\i g_2\right). 
$$
An elementary calculation shows that $[g\cdot \eps]=[\eps]$ implies that we must have $$
\begin{pmatrix} c_1 \\ c_2\end{pmatrix}=\begin{pmatrix}g^2_1 & -g^1_1 \\ g^2_2 & -g^1_2\end{pmatrix}\begin{pmatrix}g^1_0 \\ g^2_0 \end{pmatrix}=\begin{pmatrix} 0 \\ 0 \end{pmatrix}.
$$
Since
$$
\det g=g^0_2\,c_1-g^0_1\,c_2+g^0_0\,\det\begin{pmatrix}g^2_1 & -g^1_1 \\ g^2_2 & -g^1_2\end{pmatrix}=1,
$$
it follows that $g^1_0=g^2_0=0$. Therefore, the stabiliser subgroup of $[\eps]$ is a subgroup of $\R_2\rtimes \mathrm{GL}(2,\R)$. Writing $a=(a^i_j)$, we obtain
$$
\left(b\rtimes a\right) \cdot \eps=\det a^{-1}\begin{pmatrix}0 & -a^2_1-\i a^2_2 & a^1_1+\i a^1_2 \end{pmatrix},
$$
from which it follows that $[(b\rtimes a)\cdot \eps]=[\eps]$ if and only if $a^1_1=a^2_2$ and $a^1_2+a^2_1=0$, that is, $a \in \mathrm{CO}(2)$. Concluding, we have shown
$$
\mathrm{SL}(3,\R)/\left(\R_2\rtimes\mathrm{CO}(2)\right)\simeq \mathbb{CP}_2\setminus \mathbb{RP}_2
$$
and the projection map is 
$$
\mu:\mathrm{SL}(3,\R) \to \mathbb{CP}_2\setminus \mathbb{RP}_2, \quad \begin{pmatrix} g_0 & g_1 & g_2 \end{pmatrix} \mapsto [g_0\wedge (g_1+\i g_2)],
$$
where we use $\R_3\simeq \Lambda^2(\R^3)$. 
\begin{rmk}
We have only shown that $Z=\mathrm{SL}(3,\R)/\left(\R_2\times \mathrm{CO}(2)\right)$ is diffeomorphic to $\mathbb{CP}_2\setminus\mathbb{RP}_2$. Since $Z$ carries an integrable almost complex structure $J$, we may ask if $(Z,J)$ is biholomorphic to $\mathbb{CP}_2\setminus\mathbb{RP}_2$ equipped with the standard complex structure. This is indeed the case, see~\cite[Prop.~3]{MR3144212}. As a consequence of this result one can prove that the conformal connections on the $2$-sphere whose (unparametrised) geodesics are the great circles are in one-to-one correspondence with the smooth quadrics in $\mathbb{CP}_2\setminus\mathbb{RP}_2$, see~\cite[Cor.~2]{MR3144212}.
\end{rmk}
\begin{rmk}
In fact~\cite{MR1979367}, if $\mathfrak{p}$ is a projective structure on the $2$-sphere, all of whose geodesics are simple closed curves, then $Z$ can be compactified and the compactification is biholomorphic to $\mathbb{CP}_2$. This allowed Lebrun and Mason to prove that there is a nontrivial moduli space of such projective structures on the $2$-sphere.    
\end{rmk}
We will show next that $Y$ is a submanifold of $F(\C_3)$. Clearly, the action of $\mathrm{SL}(3,\R)$ on the space $F(\C_3)$ of complete complex flags is not transitive, there is however an open orbit. Let $F(\C_3)^*$ denote the $\mathrm{SL}(3,\R)$ orbit of the flag
$$
(\ell,\Pi)=\left(\mathbb{C}\{\eps_1\},\mathbb{C}\{\eps_1,\eps_2\}\right),
$$
where
$$
\eps_1=\begin{pmatrix} 0&-\i&1\end{pmatrix}, \quad \eps_2=\begin{pmatrix} 1&0&0\end{pmatrix}.\quad %\eps_3=\begin{pmatrix} 0&1&0\end{pmatrix}.
$$
We already know that the stabiliser subgroup $\mathrm{G}_0$ of $(\ell,\Pi)$ must be a subgroup of $\R_2\rtimes \mathrm{CO}(2)$. For $b\rtimes a \in \R_2\rtimes \mathrm{CO}(2)$ we write
$$
b\rtimes a=\begin{pmatrix} \frac{1}{x^2+y^2} & b_1 & b_2 \\ 0 & x & y \\  0 & -y & x\end{pmatrix},
$$
with $x^2+y^2>0$. We compute
$$
\eps_2 \cdot (b\rtimes a)=\begin{pmatrix} x^2+y^2 & -xb_1-yb_2 & -xb_2+yb_1\end{pmatrix}
$$ 
which is easily seen to lie in the complex linear span of $\eps_1,\eps_2$ if and only if $b_1=b_2=0$, hence
$$
\mathrm{SL}(3,\R)/\mathrm{CO}(2)\simeq F(\C_3)^*
$$
and the projection map is
$$
\tau : \mathrm{SL}(3,\R) \to F(\C_3) \quad \begin{pmatrix} g_0 & g_1 & g_2 \end{pmatrix} \mapsto \left(\C\{\eps_1\},\C\{\eps_1,\eps_2\}\right),
$$
with 
$$
\eps_1=g_0\wedge (g_1+\i g_2), \qquad \eps_2=g_1\wedge g_2.
$$
Since $F(\C_3)$ is real six-dimensional and since $\dim \mathrm{SL}(3,\R)-\dim \mathrm{CO}(2)=6$, it follows that $F(\C_3)^*$ is open. \section{The variational equations}

By construction, a conformal structure $[g]$ on the oriented projective surface $(\Sigma,\mathfrak{p})$ is a section of $Z \to \Sigma$. Here we will show that every conformal structure $[g]$ admits a natural lift $\widetilde{[g]} : \Sigma \to Y$. In doing so we recover the functional $\mathcal{E}_{\mathfrak{p}}$ from a different viewpoint, which simplifies the computation of its variational equations. We start with recalling the bundle of complex linear coframes of a Riemann surface. 

\subsection{The bundle of complex linear coframes}\label{subsec:confcon}

Let $\Sigma$ be an oriented surface equipped with a conformal structure $[g]$, so that $\Sigma$ inherits the structure of a Riemann surface whose integrable almost complex structure will be denoted by $J$. The bundle of complex-linear coframes of $(\Sigma,[g])$ is the $\mathrm{GL}(1,\C)$-subbundle $F^+_{[g]}$ of $F^+$ consisting of those coframes that are complex-linear with respect to $J$ and the complex structure obtained on $\R^2$ via the standard identification $\R^2\simeq \C$. Of course, via the isomorphism $\mathrm{CO}(2)\simeq \mathrm{GL}(1,\C)$, we may equivalently think of $F^+_{[g]}$ as consisting of those coframes in $F^{+}$ that are angle preserving with respect to $[g]$ and the standard conformal inner product on $\R^2$. 

Recall that a principal $\mathrm{CO}(2)$-connection $\varphi$ on $F^+_{[g]}$ is called~\textit{torsion-free} if it satisfies
$$
\d \omega=-\varphi\wedge\omega,
$$
where here we think of the tautological $\R^2$-valued $1$-form $\omega$ on $F^+_{[g]}$ as taking values in $\C$ and the connection taking values in the Lie algebra of $\mathrm{CO}(2)\simeq\mathrm{GL}(1,\C)$, that is, $\C$. The curvature $\Phi$ of $\varphi$ is a $(1,\! 1)$-form on $\Sigma$ whose pullback to $F^{+}_{[g]}$ can be written as 
$$
\d \varphi=R\,\omega\wedge\ov{\omega}
$$
for some unique complex-valued function $R$ on $F^+_{[g]}$. By definition of $F^{+}_{[g]}$, a complex-valued $1$-form on $\Sigma$ is a $(1,\! 0)$-form with respect to $J$ if and only if its pullback to $F^{+}_{[g]}$ is a complex multiple of $\omega$. A consequence of this is the following elementary lemma whose proof we omit: 
\begin{lem}\label{lem:struceqcplxconriemsurf}
A complex-valued function $f$ on $F^{+}_{[g]}$ represents a section of $K_{\Sigma}^{m}\otimes \ov{K_{\Sigma}^{n}}$ if and only if there exist complex-valued functions $f^{\prime}$ and $f^{\prime\prime}$ on $F^+_{[g]}$ so that
$$
\d f=f^{\prime}\omega+f^{\prime\prime}\ov{\omega}+fm\varphi+fn\ov{\varphi}. 
$$
\end{lem}
\begin{rmk}
Here $K_{\Sigma}=T^*_{\C}\Sigma^{1,0}$ denotes the canonical bundle of $(\Sigma,J)$, $K_{\Sigma}^m$ its $m$-th tensorial power and $\ov{K^n_{\Sigma}}$ the  conjugate bundle of the $n$-th tensorial power of $K_{\Sigma}$.  
As usual, we we let $\nabla_\varphi$ denote the connection induced by $\varphi$ on $K_{\Sigma}^{m}\otimes \ov{K_{\Sigma}^{n}}$ and by $\nabla^{\prime}_\varphi$ its $(1,\! 0)$-part and by $\nabla^{\prime\prime}_\varphi$ its $(0,\! 1)$-part. Of course, if $s$ is the section of $K_{\Sigma}^{m}\otimes \ov{K_{\Sigma}^{n}}$ represented by $f$, then $\nabla^{\prime}_\varphi s$ is represented by $f^{\prime}$ and $\nabla^{\prime\prime}_\varphi s$ is represented by $f^{\prime\prime}$.  
\end{rmk}
Lemma~\ref{lem:struceqcplxconriemsurf} implies that $\varphi$ may also be thought of as the connection form of the connection induced by $\varphi$ on $K_{\Sigma}^{*}$. Therefore, the first Chern class of $K_{\Sigma}^{*}$ is
$$
c_1(K_{\Sigma}^{*})=\left[\frac{\i}{2\pi}\Phi\right]
$$
and hence if $\Sigma$ is compact, we obtain
\begin{equation}\label{eq:gaussbonnet}
\int_{\Sigma}\i\, \Phi=2\pi\chi(\Sigma),
\end{equation}
where $\chi(\Sigma)$ denotes the Euler-characteristic of $\Sigma$. 

\subsection{Submanifold theory in the twistor space}
We are interested in co-di\-men\-sion two submanifolds of $Z$ arising as images of sections of $Z \to \Sigma$. The second order theory of such submanifolds is summarised in the following:

\begin{lem}\label{lem:lift}
Let $[g] : \Sigma \to Z$ be a conformal structure on $(\Sigma,\mathfrak{p})$. Then there exists a lift $\widetilde{[g]} : \Sigma \to Y$ covering $[g]$ so that the pullback-bundle $p : \projb_{[g]}^{\prime}=\widetilde{[g]}^*\projb \to \Sigma$ is isomorphic to the $\mathrm{CO}(2)$-bundle of complex linear coframes $F^+_{[g]}$ of $(\Sigma,[g])$ and so that on $\projb^{\prime}_{[g]}\simeq F^+_{[g]}$ we have
$$
\zeta_2=2\ov{a}\,\ov{\zeta_1}, \quad \zeta_3=k\zeta_1+2\ov{q}\ov{\zeta_1},
$$
for unique complex-valued functions $a,k,q$ on $\projb^{\prime}_{[g]}$. 
\end{lem}
%\begin{rmk}
%Here $K_{\Sigma}$ denotes the canonical bundle of $\Sigma$ with respect to the complex structure on $\Sigma$ induced by $[g]$ and the orientation.
%\end{rmk}
\begin{proof}
First recall that in Lemma~\ref{cplx3fold} we have defined
$$
\zeta_1=\theta^1_0+\i\theta^2_0,\quad \zeta_2=\left(\theta^1_1-\theta^2_2\right)+\i\left(\theta^1_2+\theta^2_1\right),\quad \zeta_3=\theta^0_1+\i\theta^0_2,
$$
where $\theta=(\theta^i_j)$ is the Cartan connection of $(\Sigma,\mathfrak{p})$. 

Let now $[g] : \Sigma \to Z$ be a conformal structure on $(\Sigma,\mathfrak{p})$ and let $p : \projb_{[g]}=[g]^*\projb\to\Sigma$ denote the pullback of the bundle $\mu : \projb\to Z$, that is, 
$$
\projb_{[g]}=\left\{(p,u) \in \Sigma\times \projb\,|\, [g](p)=\mu(u)\right\}.
$$ 
Since $\projb_{[g]}$ is $6$-dimensional, two of the components of $\theta$ become linearly dependent when pulled back to $\projb_{[g]}$. Clearly, these components must be among the $1$-forms that are semibasic for $\mu$. Recall that these forms are spanned by $\zeta_1,\zeta_2$ and their complex conjugates. However, since $[g]$ is a section of $Z \to \Sigma$ and since the $1$-forms that are semibasic for the projection $\pi : \projb \to \Sigma$ are spanned by $\zeta_1,\ov{\zeta_1}$, it follows that $\zeta_1\wedge\ov{\zeta_1}$ is non-vanishing on $\projb_{[g]}$. Therefore, on $\projb_{[g]}$ we have the relation
\begin{equation}\label{defa}
\zeta_2=2\ov{a}\ov{\zeta_1}+c\zeta_1 
\end{equation}
for unique complex-valued functions $a,c$. From the equivariance properties of $\zeta_1,\zeta_2$ under the $\R_2\rtimes \mathrm{CO}(2)$-right action~\eqref{rightacttwist}, we obtain that for all $u \in \projb_{[g]}$ and $z \rtimes \rot \in \R_2\rtimes \mathrm{CO}(2)$ we have
$$
c(u\cdot z\rtimes \rot)=r^3\e^{\i\phi}c(u)+ r^2 z
$$
and
\begin{equation}\label{righttrans2}
a(u\cdot z\rtimes \rot)=r^3\e^{-3\i\phi}a(u). 
\end{equation}
It follows that the equation $c=0$ defines a locus that corresponds to a section $\widetilde{[g]} : \Sigma \to Y$ covering $[g]$. On the pullback bundle $\projb^{\prime}_{[g]}=\widetilde{[g]}^*\projb$, where
$$
\projb^{\prime}_{[g]}=\left\{(p,u) \in \Sigma\times \projb\,|\, \widetilde{[g]}(p)=\tau(u)\right\},
$$
we obtain
\begin{equation}\label{someid}
\zeta_2=2\ov{a}\ov{\zeta_1}. 
\end{equation}   
Since $\projb^{\prime}_{[g]}$ is $4$-dimensional, two of the remaining components of $\theta$ become linearly dependent when pulled back to $\projb^{\prime}_{[g]}$. Since the $1$-forms that are semibasic for the projection $\tau : \projb \to Y$ are spanned by $\zeta_1,\zeta_2,\zeta_3$ and their complex conjugates, it follows as before that
\begin{equation}\label{someid2}
\zeta_3=k\zeta_1+2\ov{q}\,\ov{\zeta_1}
\end{equation}
for unique complex-valued functions $k,q$. %where we abbreviate
%\begin{equation}\label{someeq53}
%s=\left(\frac{1}{4}|a|^2+\frac{1}{2}c-\ov{c}\right).
%\end{equation}

Now recall that Cartan's bundle $\pi : \projb\to \Sigma$ is isomorphic to $F^+\times \R_2\to \Sigma$ equipped with the $\mathrm{G}$-right action~\eqref{rightactioncartbundleisom}. Therefore, $\projb_{[g]} \to \Sigma$ is isomorphic to $F^+_{[g]} \times \R_2 \to \Sigma$ and consequently, the bundle $\projb^{\prime}_{[g]} \to \Sigma$ is isomorphic to $F^+_{[g]} \to \Sigma$. 
\end{proof}

We also obtain:
\begin{lem}\label{lem:struceqexcon}
The functions $a,k,q$ and the $1$-form $\varphi$ satisfy the following structure equations on $P^{\prime}_{[g]}\simeq F^+_{[g]}$ 
\begin{align}
\label{eq:struceqforr}\d a&=a^{\prime}\zeta_1-q \ov{\zeta_1}+2a\varphi-a\ov{\varphi},\\
\d k&=k^{\prime}\zeta_1+k^{\prime\prime}\ov{\zeta_1}+k\varphi+k\ov{\varphi},\\
\d q&=q^{\prime}\zeta_1+\frac{1}{2}\left(\ov{L}+\ov{k^{\prime\prime}}-2\ov{q}a\right)\ov{\zeta_1}+2q \varphi,\\
\label{eq:struceqforA}\d \varphi&=\left(|a|^2+\frac{1}{2}k-\ov{k}\right)\zeta_1\wedge\ov{\zeta_1}
\end{align}
for unique complex-valued functions $r^{\prime},k^{\prime},k^{\prime\prime}$ and $q^{\prime}$ on $P^{\prime}_{[g]}$. 
\end{lem}
\begin{proof}
We will only verify the structure equation for $a$ as the other structure equations are derived in an entirely analogous fashion. The structure equations~\eqref{righttranskaeinf} and~\eqref{someid} gives
\begin{align*}
\d\ov{\zeta_2}&=\d(2a\zeta_1)=2\d a\wedge\zeta_1+a\,\d \zeta_1=-\zeta_1\wedge \d a +2a\left(\zeta_1\wedge \varphi+\frac{1}{2}\ov{\zeta_1}\wedge\zeta_2\right)\\
&=-\ov{\zeta_1}\wedge\ov{\zeta_3}+\ov{\zeta_2}\wedge \ov{\varphi}-\ov{\zeta_2}\wedge \varphi\\
&=2q\zeta_1\wedge\ov{\zeta_1}+2a\zeta_1\wedge\ov{\varphi}-2a\zeta_1\wedge \varphi,
\end{align*}
where we have used~\eqref{someid2}. Equivalently, we obtain
$$
0=\left(\d a+q\ov{\zeta_1}-2a \varphi+a\ov{\varphi}\right)\wedge\zeta_1,
$$
which implies~\eqref{eq:struceqforr}. Finally, the structure equation~\eqref{eq:struceqforA} for $\varphi$ is an immediate consequence of~\eqref{righttranskaeinf},~\eqref{someid} and~\eqref{someid2}.
\end{proof}

As we will see next, the functions $a,q,k$ on $P^{\prime}_{[g]}$ satisfy certain equivariance properties with respect to the $\mathrm{CO}(2)$-right action on $P^{\prime}_{[g]}$ and hence represent sections of complex line bundles associated to $p : P^{\prime}_{[g]}\to \Sigma$.
\begin{ppn}\label{ppn:liftsplit}
The choice of a conformal structure $[g]$ on $(\Sigma,\mathfrak{p})$ determines the following objects:
\begin{itemize}
\item[(i)] A torsion-free connection $\varphi$ on the bundle of complex-linear coframes of $(\Sigma,[g])$; 
\item[(ii)] A section $\alpha$ of $K^2_{\Sigma}\otimes \ov{K_{\Sigma}^{*}}$ that is represented by $a$.
\item[(iii)] A quadratic differential $Q$ on $\Sigma$ that is represented by $q$;
\item[(iv)] A $(1,\! 1)$-form $\kappa$ on $\Sigma$ that is represented by $k$.
\end{itemize}
Moreover, the quadratic differential $Q$ satisfies
\begin{equation}\label{eq:quaddifeq}
Q=-\nabla^{\prime\prime}_\varphi\alpha. 
\end{equation}
\end{ppn}
\begin{proof}
By construction of $P^{\prime}_{[g]}\simeq F^+_{[g]}$, a complex-valued $1$-form on $\Sigma$ is a $(1,\! 0)$-form for the complex structure $J$ induced by $[g]$ and the orientation if and only if its $p$-pullback to $\projb^{\prime}_{[g]}$ is a complex multiple of $\zeta_1$. Since 
$$
\left(R_{\rot}\right)^*\zeta_1=\frac{1}{r^3}\e^{\i\phi}\zeta_1
$$
it follows that the sections of $K_{\Sigma}^2$ are in one-to-one correspondence with the com\-plex-valued functions $f$ on $\projb^{\prime}_{[g]}$ satisfying 
$$
\left(R_{\rot}\right)^*f=r^3\e^{-\i\phi}r^3\e^{-\i\phi}f=r^6\e^{-2\i\phi} f.
$$
Likewise, it follows that the sections of $K_{\Sigma}^2\otimes \ov{K_{\Sigma}^{*}}$ are in one-to-one correspondence with the complex-valued functions $f$ on $\projb^{\prime}_{[g]}$ satisfying 
$$
\left(R_{\rot}\right)^*f=r^3\e^{-\i\phi}r^3\e^{-\i\phi}\ov{r^{-3}\e^{\i\phi}} f=r^3\e^{-3\i\phi}f
$$ 
and that the sections of $K_{\Sigma}\otimes \ov{K_{\Sigma}}$ are in one-to-one correspondence with the complex valued functions $f$ on $\projb^{\prime}_{[g]}$ satisfying
$$
\left(R_{\rot}\right)^*f=r^3\e^{-\i\phi}r^3\e^{\i\phi}f=r^6f.
$$
From~\eqref{someid2} and~\eqref{righttranskae} we obtain that for all $u \in \projb^{\prime}_{[g]} $ and $\rot \in \mathrm{CO}(2)$ 
\begin{align*}
k(u\cdot \rot)&=r^6k(u),\\ 
q(u\cdot \rot)&=r^6\e^{-2\i\phi}q(u).
\end{align*}
These equations imply that there exists a unique quadratic differential $Q$ on $\Sigma$ that is represented by $q$ and a unique $(1,\! 1)$-form $\kappa$ on $\Sigma$ that is represented by $k$. Furthermore,~\eqref{righttrans2} implies that there exists a unique section $\alpha$ of $K_{\Sigma}^2\otimes \ov{K_{\Sigma}^{*}}$ that is represented by $a$. 

It follows from the properties (ii) and (iii) of the Cartan connection that $\varphi$ is a connection $1$-form on the $\mathrm{CO}(2)$-bundle $\projb^{\prime}_{[g]} \to \Sigma$. Its pushforward under the bundle isomorphism $\projb^{\prime}_{[g]} \to F^+_{[g]}$ is then a $\mathrm{CO}(2)$-connection on $F^+_{[g]}$ which -- by abuse of notation -- we denote by $\varphi$ as well. The structure equation~\eqref{eq:struceqforA} implies that $\varphi$ is torsion-free.
 
Finally, the identity $Q=-\nabla^{\prime\prime}_\varphi\alpha$ is an immediate consequence of the structure equation~\eqref{eq:struceqforr} and Lemma~\ref{lem:struceqcplxconriemsurf}.   
\end{proof}

%Clearly, if we identify $\R^2\simeq \C$ in the usual way, an orientation preserving $[g]$-conformal coframe at $p \in \Sigma$ is a linear map $T_p\Sigma \to \R^2$ that is complex linear with respect to the complex structure $J$ on $\Sigma$ indcued by $[g]$ and the orientation and the standard complex structure on $\C$. Recall~\eqref{righttranskae} that the complex-valued $1$-form $\zeta_1$ satisfies $(R_{\rot})^*\zeta_1=r^{-3}\e^{\i\phi}\zeta_1$ so that $\zeta_1\circ\ov{\zeta_1}$ defines a Riemannian metric on $\Sigma$ that is well-defined up to scale. The resuling conformal structure is of course $[g]$. Likewise, $\frac{\i}{2}\zeta_1\wedge\ov{\zeta_1}$ defines an orientation on $\Sigma$ which is the same as the given orientation on $\Sigma$. \com{Complete remaining parts of proof.}
We call a map $\psi : (M,g) \to (N,h)$ between two pseudo-Riemannian manifolds~\textit{weakly conformal} if $\psi^*h=f g$ for some smooth function $f$ on $M$. Note that we do not require $f$ to be positive. Two immediate consequences of Proposition~\ref{ppn:liftsplit} are: 
\begin{cor}\label{somecor}
Let $[g]$ be a conformal structure on $(\Sigma,\mathfrak{p})$. Then the lift $\widetilde{[g]} : (\Sigma,[g]) \to (Y,\Ym)$ is weakly conformal if and only if $Q\equiv 0$. Furthermore, the image of $[g] : \Sigma \to Z$ is a holomorphic curve if and only if $\alpha\equiv  0$. In particular, if $[g](\Sigma) \subset Z$ is a holomorphic curve, then $\widetilde{[g]}(\Sigma)\subset Y$ is a holomorphic contact curve.  
\end{cor}
\begin{rmk}
Here we call a holomorphic curve $\Sigma\subset Y$ a~\textit{contact curve} if its tangent bundle is contained in the (holomorphic) contact structure of $Y$.
\end{rmk} 
\begin{proof}[Proof of Corollary~\ref{somecor}]
By construction, the metric $\Ym$ has the property that its pullback to $\projb$ is
$$
\tau^*\Ym=\frac{1}{2}\left(\zeta_1\circ \ov{\zeta_3}+\zeta_3\circ\ov{\zeta_1}+\zeta_2\circ\ov{\zeta_2}\right).
$$
Therefore, from~\eqref{someid} and~\eqref{someid2} it follows that 
\begin{equation}\label{pullbackmetric}
%\zeta\circ\left(\ov{c}\ov{\zeta}+q\zeta\right)+\left(c\zeta+\ov{q}\ov{\zeta}\right)\circ \ov{\zeta}+|a|^2\zeta\circ\ov{\zeta}=
p^*\left(\widetilde{[g]}^*\Ym\right)=\frac{1}{2}\left(4|a|^2+(k+\ov{k})\right)\zeta_1\circ\ov{\zeta_1}+q\,\zeta_1\circ\zeta_1+\ov{q}\,\ov{\zeta_1}\circ\ov{\zeta_1}. 
\end{equation}
Since a complex-valued $1$-form on $\Sigma$ is a $(1,\! 0)$-form for the complex structure defined by $[g]$ and the orientation if and only if its $p$-pullback to $\projb^{\prime}_{[g]}$ is a complex multiple of $\zeta_1$, equation~\eqref{pullbackmetric} implies that $\widetilde{[g]}^*\Ym$ is weakly conformal to $[g]$ if and only if $q$ vanishes identically. The first claim follows.  

The second part of the claim is an immediate consequence of~\eqref{someid} and the characterisation of the complex structures on $Z,Y$ in terms of $\zeta_1,\zeta_2,\zeta_3$ and the characterisation of the holomorphic contact structure in terms of $\zeta_2=0$.  
\end{proof}
\begin{rmk}
Recall that if $\psi : (\Sigma,[g]) \to (N,h)$ is a map from a Riemann surface into a (pseudo-)Riemannian manifold, then the $(2,\! 0)$-part of the pulled back metric $\psi^*h$ is called the~\textit{Hopf differential} of $\psi$. Therefore~\eqref{pullbackmetric} implies that quadratic differential $Q$ is the Hopf differential of $\widetilde{[g]}$. 
\end{rmk}
Proposition~\ref{ppn:liftsplit} shows that for every choice of a conformal structure $[g]$ on $\Sigma$ we obtain a section $\alpha$ of $K_{\Sigma}^2\otimes \ov{K_{\Sigma}^{*}}$, as well as a connection $\varphi$ on the principal $\mathrm{GL}(1,\C)$-bundle of complex-linear coframes of $(\Sigma,[g])$. Since $K_{\Sigma}^2\otimes \ov{K_{\Sigma}^{*}}$ is a subbundle of $T^*_{\C}\Sigma^2\otimes T_{\C}\Sigma$, we may use the canonical real structure of the latter bundle to take the real part of $\alpha$. Consequently, the real part of $\alpha$ is a $1$-form on $\Sigma$ with values in $\mathrm{End}(T\Sigma)$. We have already encountered an endomorphism valued $1$-form $A_{[g]}$ whose properties we discussed in Theorem~\ref{quasilinop1}. In Corollary~\ref{uniquepoints} we have also seen that the choice of a conformal structure $[g]$ on $(\Sigma,\mathfrak{p})$ determines a unique $[g]$-conformal connection ${}^{[g]}\nabla$ so that ${}^{[g]}\nabla+A_{[g]}$ defines $\mathfrak{p}$. On the other hand, $\varphi$ also induces a $[g]$-conformal connection on $TM$ which we denote by $\nabla_{\varphi}$.    
\begin{ppn}\label{ppn:rerho=alpha}
We have:
\begin{align}
\label{eq:keyid2}\nabla_{\varphi}&={}^{[g]}\nabla,\\
\label{keyid}2\Re(\alpha)&=A_{[g]},\\
\label{keyid3}p^*\left(|A_{[g]}|^2_g\,d\mu_g\right)&=2\i|a|^2\zeta_1\wedge\ov{\zeta_1}=-\frac{\i}{2}\zeta_2\wedge\ov{\zeta_2}. 
\end{align}
\end{ppn}
%\tm{Write down proof}
Since a $[g]$-conformal connection ${}^{[g]}\nabla$ has holonomy in $\mathrm{CO}(2)$, it corresponds to a unique torsion-free principal $\mathrm{CO}(2)$-connection $\varphi$ on $F^+_{[g]}$, see for instance~\cite{MR1427757}. Before proving Proposition~\ref{ppn:rerho=alpha} it is helpful to see explicitly how the principal connection $\varphi$ is constructed from ${}^{[g]}\nabla$. The $[g]$-conformal connection ${}^{[g]}\nabla$ can be written as
\begin{equation}\label{eq:confconform}
{}^{(g,\beta)}\nabla={}^g\nabla+g\otimes \beta^{\sharp}-\beta\otimes\mathrm{Id}-\mathrm{Id}\otimes \beta, 
\end{equation}
where $g\in [g]$ and $\beta$ is a $1$-form on $M$ with $g$-dual vector field $\beta^{\sharp}$. Let $g_{ij}=g_{ji}$ be the unique real-valued functions on $F^+$ so that $\upsilon^*g=g_{ij}\omega^i\otimes \omega^j$. Let $\psi=(\psi^i_j)$ denote the Levi-Civita connection form of $g$, so that we have the structure equations. %Let furthermore $A^i_{jk}$ be the unique real-valued functions on $F^+$ representing $A_{[g]}$. 
\begin{align*}
\d\omega^i&=-\psi^i_j\wedge\omega^j,\\
\d g_{ij}&=g_{ik}\psi^k_j+g_{kj}\psi^k_i
\end{align*}
as well as
\begin{equation}\label{someeq37}
\d \psi^i_j+\psi^i_k\wedge\psi^k_j=g_{jk}K_g\omega^i\wedge\omega^k,
\end{equation}
where the real-valued function $K_g$ on $F^+$ is (the pullback of) the Gauss curvature of $g$. Therefore, writing $\upsilon^*\beta=b_i\omega^i$ for real-valued functions $b_i$ on $F^+$, the connection $1$-form of~\eqref{eq:confconform} is
$$
\eta^i_j=\psi^i_j+\left(b_kg^{ki}g_{jl}-\delta^i_jb_l-\delta^i_lb_j\right)\omega^l,
$$ 
where the real-valued functions $g^{ij}=g^{ji}$ on $F^+$ satisfy $g^{ik}g_{kj}=\delta^i_j$. The equivariance properties of the functions $b_i$ imply that there exist unique real-valued functions $b_{ij}$ on $\F$ so that
\begin{equation}\label{someeq38}
\d b_i=b_j\psi^j_i+b_{ij}\omega^j. 
\end{equation}
From the equivariance properties of the functions $g_{ij}$ it follows that the conditions $g_{11}= g_{22}$ and $g_{12}=0$ define a reduction of $\upsilon : F^+ \to \Sigma$ to the $\mathrm{CO}(2)$-subbundle of complex linear coframes of $F^+_{[g]} \to \Sigma$ of $(\Sigma,[g])$. On $F^+_{[g]}$ we obtain
$$
0=\d g_{12}=g_{11}\psi^1_2+g_{12}\psi^2_2+g_{12}\psi^1_1+g_{22}\psi^2_1=g_{11}(\psi^1_2+\psi^2_1)
$$
and hence $\psi^2_1=-\psi^1_2$. Likewise, we have
\begin{align*}
0&=\d g_{11}-\d g_{22}=2\left(g_{11}\psi^1_1+g_{12}\psi^2_1\right)-2\left(g_{12}\psi^1_2+g_{22}\psi^2_2\right)\\
&=2g_{11}\left(\psi^1_1-\psi^2_2\right)
\end{align*}
so that $\psi^1_1=\psi^2_2$. Idenfifying $\R^2\simeq \C$, we may think of $\omega=(\omega^i)$ as taking values in $\C$. If we define $\varphi:=\frac{1}{2}\left(\eta^1_1+\eta^2_2\right)+\frac{\i}{2}\left(\eta^2_1-\eta^1_2\right)$, 
we obtain
\begin{align}\label{eq:pullbackconformframe}
\varphi=\left(\psi^1_1-b_1\omega^1-b_2\omega^2\right)+\i\left(\psi^2_1+b_2\omega^1-b_1\omega^2\right)
\end{align}
Using this notation the first structure equation can be written in complex form
$$
\d \omega=-\varphi\wedge\omega, 
$$
hence $\varphi$ defines a torsion-free principal $\mathrm{CO}(2)$-connection on $F^+_{[g]}$. 

\begin{proof}[Proof of Proposition~\ref{ppn:rerho=alpha}] Without loosing generality, we can assume that $\mathfrak{p}$ is defined by ${}^{[g]}\nabla+A_{[g]}$ for some $[g]$-conformal connection ${}^{[g]}\nabla$ and some $1$-form $A_{[g]}$ having all the properties of Theorem~\ref{quasilinop1}. 
Recall~\eqref{cartconpref} that the choice of a representative connection $\nabla \in \mathfrak{p}$ gives an identification $\projb \simeq F^+\rtimes\R_2$ of Cartan's bundle so that the Cartan connection form becomes
\begin{equation}\label{eq:petra1}
\theta=\begin{pmatrix} -\frac{1}{3}\tr \eta-\xi \omega & \d \xi-\xi\eta-(S\omega)^t-\xi\omega\xi\\  \omega & \eta-\frac{1}{3}\mathrm{I}\tr \eta+\omega\xi  \end{pmatrix}.
\end{equation}
We will construct Cartan's connection for the representative connection 
\begin{equation}\label{eq:yetanotherid}
{}^{(g,\beta)}\nabla+A_{[g]}={}^g\nabla+g\otimes \beta^{\sharp}-\beta\otimes\mathrm{Id}-\mathrm{Id}\otimes \beta+A_{[g]}. 
\end{equation}
Let $A^i_{jk}$ denote the real-valued functions on $F^+$ representing $A_{[g]}$. In particular, we have 
\begin{equation}\label{eq:petra3}
A^i_{jk}=A^i_{kj}\quad \text{and} \quad A^l_{il}=0.
\end{equation}
On $F^+$ the connection form of~\eqref{eq:yetanotherid} is given by
\begin{equation}\label{eq:petra2}
\eta^i_j=\psi^i_j+\left(b_kg^{ki}g_{jl}-\delta^i_jb_l-\delta^i_lb_j+A^i_{jl}\right)\omega^l,
\end{equation}
By definition, the pullback bundle $P_{[g]}$ is the subbundle of $F^+\times \R_2$ defined by the equations $g_{11}=g_{22}$ and $g_{12}=0$. Now on $P_{[g]}\simeq F^+_{[g]}\times \R_2$ we have $\psi^2_1=-\psi^1_2$ and $\psi^1_1=\psi^2_2$. Using~\eqref{eq:petra1},~\eqref{eq:petra3} and~\eqref{eq:petra2} we compute
\begin{align*}
\zeta_2&=(\theta^1_1-\theta^2_2)+\i\left(\theta^1_2+\theta^2_1\right)\\
&=\psi^1_1-\psi^2_2+\left(\xi_1+2A^1_{11}\right)\omega^1+\left(-\xi_2-2A^2_{22}\right)\omega^2\\
&\phantom{=}+\i\left(\psi^1_2+\psi^2_1+\left(\xi_2-2A^2_{22}\right)\omega^1+\left(\xi_1-2A^1_{11}\right)\omega^2\right)\\
%&=\frac{1}{2}\left(A^1_{11}+A^2_{12}+2\xi_1+\i(A^2_{11}+A^2_{22}+2\xi_2)\right)(\omega^1+\i \omega^2)\\
%&\phantom{=}+\frac{1}{2}\left(A^1_{11}-3A^2_{12}+\i(3A^2_{11}-A^2_{22})\right)\left(\omega^1-\i\omega^2\right)\\
&=2\ov{a}\ov{\zeta_1}+c\zeta_1,
\end{align*}
where
\begin{align}
\label{eq:keyidupstairs}a&=A^1_{11}+\i A^2_{22},\\
c&=\xi_1+\i \xi_2
\end{align}
and we have used that on $F^+_{[g]}$  
$$
\delta_{il}A^l_{jk}=\delta_{jl}A^l_{ik},
$$
which follows from Theorem~\ref{quasilinop1} (vi). Recall that $P^{\prime}_{[g]}$ was defined by the equation $c=0$. Hence on $P^{\prime}_{[g]}\simeq F^+_{[g]}$ the function $\xi$ vanishes identically. Using this we compute
$$
\varphi=-\frac{1}{2}\left(3\theta^0_0+\i\left(\theta^1_2-\theta^2_1\right)\right)=\psi^1_1-b_1\omega^1-b_2\omega^2+\i\left(\psi^2_1+b_2\omega^1-b_1\omega^2\right).
$$
This is precisely~\eqref{eq:pullbackconformframe}. It follows that the connection defined by $\varphi$ is the same as the induced torsion-free connection on $F^+_{[g]}$ by ${}^{[g]}\nabla$. This proves~\eqref{eq:keyid2}. 

Suppose $x=(x^i) : U \to \R^2$ are local orientation preserving $[g]$-isothermal coordinates on $\Sigma$ and write $z=(x^1+\i x^2)$. Applying the exterior derivative to $z$ we obtain a local section $\tilde{z} : U \to F^+_{[g]}$ so that
$$
A_{[g]}=\tilde{z}^*A^i_{jk}\d x^j\otimes \d x^k\otimes \frac{\partial}{\partial x^i}.
$$
By definition of $\alpha$ we have
$$
\alpha=\tilde{z}^*a\,\d z\otimes\d z\otimes \otimes \frac{\partial}{\partial \ov{z}},
$$
hence~\eqref{keyid} is an immediate consequence of~\eqref{eq:keyidupstairs}. 

Finally, in our coordinates we obtain
$$
|A_{[g]}|_g^2\,d\mu_g=4|a|^2\d x^1\wedge \d x^2, 
$$
so that $p^*\left(|A_{[g]}|_g^2\,d\mu_g\right)=2\i|a|^2\zeta_1\wedge\ov{\zeta_1}=-\frac{\i}{2}\zeta_2\wedge\ov{\zeta_2}$, as claimed.
\end{proof}

Note that $A_{[g]}$ vanishes identically if and only if $\alpha$ vanishes identically. Therefore, as an immediate consequence of Proposition~\ref{ppn:rerho=alpha}, Corollary~\ref{uniquepoints} and Corollary~\ref{somecor}, we obtain an alternative proof of~\cite[Theorem 3]{MR3144212} (see also~\cite{MR3043749} for a `generalisation' to higher dimensions):
\begin{thm}\label{weylmetri}
A conformal structure $[g]$ on $(\Sigma,\mathfrak{p})$ is preserved by a conformal connection defining $\mathfrak{p}$ if and only if the image of $[g] : \Sigma \to Z$ is a holomorphic curve. 
\end{thm}
\begin{rmk}
Locally the bundle $Z \to \Sigma$ always admits sections having holomorphic image and therefore every torsion-free connection on $T\Sigma$ is locally projectively equivalent to a conformal connection (see~\cite{MR3144212} for additional details). 
\end{rmk}
%\begin{defn}
%We call a conformal structure $[g]$ on $\Sigma$ \textit{spacelike} or \textit{timelike} if $\widetilde{[g]}^*\Ym$ is positive definite or negative definite. Furthermore, we call a spacelike or timelike conformal structure~\textit{minimal} if the mean curvature vector of the submanifold $\widetilde{[g]}(\Sigma)\subset Y$ vanishes identically. 
%\end{defn}

\subsection{Derivation of the variational equations}

Applying a technique from~\cite{MR772125}, we compute the variational equations for the functional $\mathcal{E}_{\mathfrak{p}}$. For a compact domain $\Omega\subset \Sigma$ and a section $[g] : \Sigma \to Z$ we write
$$
\mathcal{E}_{\mathfrak{p},\Omega}([g])=\int_{\Omega}|A_{[g]}|^2_gd\mu_g. 
$$ 
\begin{defn}
We say $[g]$ is an $\mathcal{E}_{\mathfrak{p}}$-critical point or that $[g]$~\textit{is extremal for the projective structure} $\mathfrak{p}$ if for every compact $\Omega\subset \Sigma$ and for every smooth variation $[g]_t : \Sigma \to Z$ with support in $\Omega$, we have
$$
\frac{\d}{\d t}\bigg|_{t=0}\mathcal{E}_{\mathfrak{p},\Omega}([g]_t)=0. 
$$
\end{defn}
Using this definition we obtain:
\begin{thm}\label{thm:weakconf}
Let $(\Sigma,\mathfrak{p})$ be an oriented projective surface. A conformal structure $[g]$ on $\Sigma$ is extremal for $\mathfrak{p}$ if and only if $\widetilde{[g]} : (\Sigma,[g]) \to (Y,\Ym)$ is weakly conformal.  
\end{thm}

\begin{proof}
Let $[g] : \Sigma \to Z$ be a conformal structure and $[g]_t : \Sigma \to Z$ a smooth variation of $[g]$ with support in some compact set $\Omega\subset \Sigma$ and with $|t|<\eps$. We consider the submanifold of $\Sigma\times\projb\times (-\eps,\eps)$ defined by
$$
\projb^{\prime}_{[g]_t}=\left\{(p,u,t_0) \in \Sigma\times\projb\times (-\eps,\eps)\,|\,(p,u)\in \projb^{\prime}_{[g]_{t_0}}
\right\}
$$
and denote by $\iota_{[g]_t} : \projb^{\prime}_{[g]_t} \to \Sigma \times \projb \times (-\eps,\eps)$ the inclusion map. On $\Sigma\times \projb\times (-\eps,\eps)$ we define the real-valued $2$-form
$$
\mathsf{A}=-\frac{\i}{2}\zeta_2\wedge\ov{\zeta_2},%.-\d t\wedge\left(\partial_t \inc\left(\zeta_2\wedge\ov{\zeta_2}\right)\right)\Bigg],
$$
where, by abuse of notation, we write $\zeta_2$ for the pullback of $\zeta_2$  to $\Sigma\times \projb\times (-\eps,\eps)$. Using the structure equations~\eqref{righttranskaeinf}, we compute
\begin{equation}\label{eq:mainform}
\d\mathsf{A}=\frac{\i}{2}\left(\zeta_1\wedge\zeta_3\wedge\ov{\zeta_2}-\zeta_2\wedge\ov{\zeta_1}\wedge\ov{\zeta_3}\right).
\end{equation}
 
%where $t$ denotes the projection onto the $(-\eps,\eps)$-factor. Note that by construction $\partial_t \inc \mathsf{A}=0$.\footnote{We use the notation $X\inc \alpha$ for the interior product of the vector field $X$ with the differential form $\alpha$.} 
Now Proposition~\ref{ppn:rerho=alpha} implies
$$
f(t_0):=\left.\mathcal{E}_{\mathfrak{p},\Omega}([g]_t)\right|_{t=t_0}=\int_{\Omega}\left.\left(\left(\iota_{[g]_t}\right)^*\mathsf{A}\right)\right|_{t=t_0}.
$$
Therefore
$$
f^{\prime}(0)=\int_{\Omega}\left.\left(\mathrm{L}_{\partial_t}(\iota_{[g]_t})^*\mathsf{A}\right)\right|_{t=0}=\int_{\Omega}\left.\left(\partial_t \inc (\iota_{[g]_t})^*\d \mathsf{A}\right)\right|_{t=0},
$$
where $\mathrm{L}_{\partial_t}$ denotes the Lie-derivative with respect to the vector field $\partial_t$. It follows from the proof of Lemma~\ref{lem:lift} that on $\projb^{\prime}_{[g]_t}$ there exist complex-valued functions $a,k,q,B,C$ such that
\begin{equation}\label{id8457}
\zeta_2=2\ov{a}\ov{\zeta_1}+B\d t\quad \text{and}\quad \zeta_3=k\zeta_1+2\ov{q}\ov{\zeta_1}+C \d t  
\end{equation}
where we now write $\zeta_i$ instead of $(\iota_{[g]_t})^*\zeta_i$. Combining~\eqref{eq:mainform} with~\eqref{id8457} gives
$$
(\iota_{[g]_t})^*\d \mathsf{A}=\i\left(qB+\ov{q} \ov{B}\right) \d t\wedge\zeta_1\wedge\ov{\zeta_1}
$$
so that
\begin{equation}\label{firstvariation}
f^{\prime}(0)=\i\int_{\Omega}\left.\left(qB+\ov{q}\ov{B}\right)\zeta_1\wedge\ov{\zeta_1}\right|_{t=0}.
\end{equation}
Recall that $\left(R_{\rot}\right)^*\zeta_2=\e^{2\i\phi}\zeta_2$ and therefore, by definition, the complex-valued function $B|_{t=0}$ satisfies 
$$
\left(R_{\rot}\right)^*\left(B|_{t=0}\right)=\e^{2\i\phi}\left(B|_{t=0}\right).
$$ 
Since $\left(R_{\rot}\right)^*\zeta_1=r^{-3}\e^{\i\phi}\zeta_1$ it follows that $B|_{t=0}$ represents a section of $\ov{K_{\Sigma}}\otimes K_{\Sigma}^{*}$ with support in $\Omega$. Here $K_{\Sigma}$ denotes the canonical bundle of $\Sigma$ with respect to the complex structure induced by the orientation and $[g]=[g]_t|_{t=0}$.

It remains to show that every such section in~\eqref{id8457} with support in $\Omega$ can be realised via some variation of $[g]$. We fix a representative metric $g \in [g]$. Let $g_{ij}=g_{ji}$ be the real-valued functions on Cartan's bundle $\projb$ so that $\pi^*g=g_{ij}\theta^i_0\otimes \theta^j_0$. In particular, from the equivariance properties (ii) of the Cartan connection $\theta$ it follows that
$$
\left(R_{b\rtimes a}\right)^*\begin{pmatrix} g_{11} & g_{12} \\ g_{21} & g_{22}\end{pmatrix}=(\det a)^2a^t\begin{pmatrix} g_{11} & g_{12} \\ g_{21} & g_{22}\end{pmatrix}a.
$$
Applying property (iii) of the Cartan connection this implies the existence of unique real-valued functions $g_{ijk}=g_{jik}$ so that
\begin{equation}\label{someid558}
\d g_{ij}=-2g_{ij}\theta^0_0+g_{kj}\theta^k_i+g_{ik}\theta^k_j+g_{ijk}\theta^k_0. 
\end{equation}
Consider the following conformally invariant functions 
$$
G=\frac{(g_{11}-g_{22})+2\i g_{12}}{\sqrt{g_{11}g_{22}-(g_{12})^2}}, \quad H=\frac{g_{11}+g_{22}}{\sqrt{g_{11}g_{22}-(g_{12})^2}}.
$$
Translating~\eqref{someid558} into complex form gives the following structure equation
\begin{equation}\label{eq:cplxformstruceqbelt}
\d G=G^{\prime}\zeta_1+G^{\prime\prime}\ov{\zeta_1}+H\zeta_2+\ov{G}\left(\ov{\varphi}-\varphi\right),
\end{equation}
for unique complex-valued functions $G^{\prime},G^{\prime\prime}$ on $\projb$. Clearly, the complex-valued functions $G^{\prime}$ and $G^{\prime\prime}$ can be expressed in terms of the functions $g_{ijk}$, as $\zeta_1=\theta^1_0+\i \theta^2_0$. In order to verify~\eqref{eq:cplxformstruceqbelt} it is thus sufficient to plug in the definitions of the functions $G,H$, the definitions of the forms $\zeta_2,\varphi$ and to use
$$
\d g_{ij}=-2g_{ij}\theta^0_0+g_{kj}\theta^k_i+g_{ik}\theta^k_j \quad \text{mod}\quad \theta^1_0,\theta^2_0.
$$
While this is somewhat tedious, it is straightforward, so we omit the computation. 

Fix a section of $\ov{K_{\Sigma}}\otimes K_{\Sigma}^{*}$ with respect to $[g]$ having support in $\Omega$. Such sections are well-known to correspond to endomorphisms of $T\Sigma$ that are trace-free and symmetric with respect to $[g]$. In particular, on $\projb$ there exist real-valued functions $(B^i_j)$ representing the corresponding endomorphism. The functions satisfy
$$
B^i_i=0 \quad \text{and} \quad g_{ij}B^j_k=g_{kj}B^j_i. 
$$
as well as the equivariance property
$$
\left(R_{b\rtimes a}\right)^*\begin{pmatrix} B^1_1 & B^1_2 \\ B^2_1 & B^2_2\end{pmatrix}=a^{-1}\begin{pmatrix} B^1_1 & B^1_2 \\ B^2_1 & B^2_2\end{pmatrix}a. 
$$
We define $B=\frac{1}{2}(B^1_1-B^2_2)+\frac{\i}{2}\left(B^1_2+B^2_1\right)$, then $B$ satisfies $(R_{z\rtimes\rot})^*B=\e^{2\i\phi}B$, hence for sufficiently small $t$ we may vary $[g]$ by defining $[g]_t$ via the zero-locus of the function
$$
G_t=G-tB H.
$$
Consequently, on
$$
\projb_{[g]_t}=\left\{(p,u,t_0)\in \Sigma\times \projb\times (-\eps,\eps)\,|\,(p,u) \in \projb_{[g]_{t_0}}\right\}
$$ 
we get
\begin{align*}
0&=\d G_t=\d G-\d tB H-t\d\left(B H\right)\\
&=G^{\prime}\zeta_1+G^{\prime\prime}\ov{\zeta_1}+H\zeta_2+\ov{G}\left(\ov{\varphi}-\varphi\right)-\d tB H-t\d\left(B H\right)\\
&=G^{\prime}\zeta_1+G^{\prime\prime}\ov{\zeta_1}+H\zeta_2+t\ov{B}H\left(\ov{\varphi}-\varphi\right)-\d tB H-t\d\left(B H\right) 
\end{align*}
In particular, if we evaluate this last equation on $\left.\projb_{[g]_t}\right|_{t=0}$, we obtain
$$
0=G^{\prime}\zeta_1+G^{\prime\prime}\ov{\zeta_1}+H\zeta_2-\d t B H
$$
Since $H$ is non-vanishing on $\left.\projb_{[g]_t}\right|_{t=0}$ we must have
$$
\zeta_2=-\frac{G^{\prime}}{H}\zeta_1-\frac{G^{\prime\prime}}{H}\ov{\zeta_1}+B\d t. 
$$
Since $\projb_{[g]_t}^{\prime}$ arises by reducing $\projb_{[g]_t}$, it follows that on $\left.\projb^{\prime}_{[g]_t}\right|_{t=0}$ we obtain
$$
\zeta_2=-\frac{G^{\prime\prime}}{H}\ov{\zeta_1}+B \d t,
$$
as desired. Finally, we now know that~\eqref{firstvariation} must vanish where $B$ is any complex-valued function representing an arbitrary section of $\ov{K_{\Sigma}}\otimes K_{\Sigma}^{*}$ with support in $\Omega$. This is only possible if $q|_{t=0}$ vanishes identically. Applying Corollary~\ref{somecor} proves the claim.   
\end{proof}
\begin{rmk}
Clearly, if $[g](\Sigma) \subset Z$ is a holomorphic curve, then $\widetilde{[g]} : \Sigma \to Y$ is weakly conformal. Using the structure equations this can be seen as follows. The image $[g](\Sigma)\subset Z$ is a holomorphic curve if and only if $\alpha$ vanishes identically. However, if $\alpha$ vanishes identically, then so does $a$ and hence~\eqref{eq:struceqforr} implies that $q$ vanishes identically as well.  Consequently, every projective structure $\mathfrak{p}$ locally admits a conformal structure $[g]$ so that $\widetilde{[g]}$ is weakly conformal.  
\end{rmk}
We conclude this section by showing that in the compact case $\mathcal{E}_{\mathfrak{p}}([g])$ is -- up to a topological constant -- just the Dirichlet energy of $\widetilde{[g]} : (\Sigma,[g])\to (Y,\Ym)$.  
\begin{lem}\label{lem:dirichlet}
Let $(\Sigma,\mathfrak{p})$ be a compact oriented projective surface. Then for every conformal structure $[g] : \Sigma \to Z$ we have
$$
\int_{\Sigma}|A_{[g]}|^2_g d\mu_g=2\pi\chi(\Sigma)+\frac{1}{2}\int_{\Sigma}\tr_g \widetilde{[g]}^*\Ym\, d\mu_g,
$$
where $\chi(\Sigma)$ denotes the Euler-characteristic of $\Sigma$. 
\end{lem}
\begin{proof}
Recall from~\eqref{pullbackmetric} that
$$
p^*\left(\widetilde{[g]}^*\Ym\right)=\frac{1}{2}\left(4|a|^2+(k+\ov{k})\right)\zeta_1\circ\ov{\zeta_1}+q\,\zeta_1\circ\zeta_1+\ov{q}\,\ov{\zeta_1}\circ\ov{\zeta_1}.
$$
%where $s$ is the curvature function of the induced $[g]$-conformal connection $\psi$, that is, 
%$$
%\Psi=\d \psi=s\zeta_1\wedge\ov{\zeta_1}.
%$$ 
%In particular, we get
%$$
%\mathrm{Im}(\Psi)=-\frac{\i}{2}\left(\d\psi-\d\ov{\psi}\right)=-\frac{\i}{2}(s+\ov{s})\zeta_1\wedge\ov{\zeta_1}. 
%$$
Hence we obtain
$$
\frac{1}{2}\int_{\Sigma}\tr_g \widetilde{[g]}^*\Ym\, d\mu_g=\frac{1}{2}\int_{\Sigma}\left(4|a|^2+(k+\ov{k})\right)\frac{\i}{2}\zeta_1\wedge\ov{\zeta_1}.%=\int_{\Sigma}|A_{[g]}|^2_gd\mu_g\\
%+\int_{\Sigma}\mathrm{Im}(\Psi).
$$
Since
$$
\d \varphi=\left(|a|^2+\frac{1}{2}k-\ov{k}\right)\zeta_1\wedge\ov{\zeta_1},
$$
we get
$$
\frac{\i}{2}\left(\d\varphi-\d\ov{\varphi}\right)=\frac{1}{2}\left(4|a|^2-(k+\ov{k})\right)\frac{\i}{2}\zeta_1\wedge\ov{\zeta_1}
$$
and thus
\begin{align*}
\frac{1}{2}\int_{\Sigma}\tr_g \widetilde{[g]}^*\Ym\, d\mu_g&=\int_{\Sigma}2\i|a|^2\zeta_1\wedge\ov{\zeta_1}-\int_{\Sigma}\frac{\i}{2}(\d \varphi-\d\ov{\varphi})\\
&=\int_{\Sigma}|A_{[g]}|^2_gd\mu_g-2\pi\chi(\Sigma),
\end{align*}
where we have used~\eqref{eq:gaussbonnet} and~\eqref{keyid3}.
\end{proof}
As an obvious consequence of Lemma~\ref{lem:dirichlet} and Theorem~\ref{quasilinop1} we have the  lower bound:
\begin{thm}
Let $(\Sigma,\mathfrak{p})$ be a compact oriented projective surface. Then for every conformal structure $[g] : \Sigma \to Z$ we have
$$
\frac{1}{2}\int_{\Sigma}\tr_g \widetilde{[g]}^*\Ym\, d\mu_g\geqslant -2\pi\chi(\Sigma),
$$
with equality if and only if $\mathfrak{p}$ is defined by a $[g]$-conformal connection.
\end{thm}
\section{Existence of critical points}

Clearly, if a projective structure $\mathfrak{p}$ is defined by a $[g]$-conformal connection, then the conformal structure $[g]$ is a critical point for $\mathcal{E}_\mathfrak{p}$ and moreover an absolute minimiser. In this final section we study the projective structures for which $\mathcal{E}_{\mathfrak{p}}$ admits a critical point in some more detail. In particular, we will prove that properly convex projective structures admit critical points.

Recall that the choice of a conformal structure $[g]$ on an oriented projective surface $(\Sigma,\mathfrak{p})$ determines a torsion-free principal $\mathrm{CO}(2)$-connection $\varphi$ on the bundle $F^+_{[g]}$ of complex linear coframes of $(\Sigma,[g])$ and a section $\alpha$ of $K_{\Sigma}^2\otimes \ov{K_{\Sigma}^{*}}$. Furthermore, the conformal structure $[g]$ is extremal for $\mathcal{E}_{\mathfrak{p}}$ if and only if $\nabla^{\prime\prime}_\varphi \alpha=0$. Conversely, let $(\Sigma,[g])$ be a Riemann surface. Let $\varphi$ be a torsion-free principal $\mathrm{CO}(2)$-connection on $F^+_{[g]}$ and $\alpha$ a section of $K_{\Sigma}^2\otimes \ov{K_{\Sigma}^{*}}$. Then Proposition~\ref{ppn:liftsplit}, Proposition~\ref{ppn:rerho=alpha} and Theorem~\ref{thm:weakconf} show that the conformal structure $[g]$ is extremal for the projective structure defined by $\nabla_{\varphi}+2\Re(\alpha)$ if and only if $\nabla_{\varphi}^{\prime\prime}\alpha\equiv 0$. Since the curvature of the connection induced by $\varphi$ on the complex line bundle $E=K^2_{\Sigma}\otimes\ov{K^*_{\Sigma}}$ is a $(1,\! 1)$-form, standard results imply~(see for instance~\cite[Prop.~1.3.7]{MR909698}) that there exists a unique holomorphic line bundle structure $\ov{\partial}_E$ on $E$, so that
$$
\ov{\partial}_E=\nabla^{\prime\prime}_\varphi. 
$$
Hence the variational equation $\nabla^{\prime\prime}_\varphi\alpha=0$ just says that $\alpha$ is holomorphic with respect to $\ov{\partial}_E$. Since the line bundle $E$ has degree
$$
\deg(E)=\deg(K^2_{\Sigma})-\deg\left(K^*_{\Sigma}\right)=-3\deg\left(K^*_\Sigma\right)=-3\chi(\Sigma), 
$$
we immediately obtain:

\begin{thm}
Suppose $\mathfrak{p}$ is a projective structure on the oriented $2$-sphere $S^2$ admitting an extremal conformal structure $[g]$. Then $\mathfrak{p}$ is defined by a $[g]$-conformal connection. 
\end{thm}
\begin{proof}
Suppose $[g]$ is an extremal conformal structure of $\mathcal{E}_{\mathfrak{p}}$. From Corollary~\ref{uniquepoints} we know that $\mathfrak{p}$ is defined by ${}^{[g]}\nabla+A_{[g]}$ for some $[g]$-conformal connection ${}^{[g]}\nabla$. Since $\chi(S^2)=2$, we have $\deg(E)=-6$ and hence the only holomorphic section of $E$ is the zero-section. It follows that $\alpha$ vanishes identically and since by Proposition~\ref{ppn:rerho=alpha} we have $A_{[g]}=2\Re(\alpha)$, so does $A_{[g]}$. 
\end{proof}
\begin{rmk}
Note that the projectively flat conformal connections on $S^2$ are classified in~\cite{MR3144212}. 
\end{rmk}
From the Riemann--Roch theorem we know that the space $H^0(\Sigma,E)$ of holomorphic sections of $E$ has dimension
$$
\dim_{\C}\,H^0(\Sigma,E)\geqslant \deg(E)+1-g_{\Sigma}=5g_{\Sigma}-5,
$$
where here $g_{\Sigma}$ denotes the genus of $\Sigma$. In particular, if $\Sigma$ has negative Euler-characteristic, then $\dim_{\C} H^0(\Sigma,E)$ will have positive dimension.  

\subsection{Convex projective structures}

Recall that a flat projective surface $(\Sigma,\mathfrak{p})$ has the property that $\Sigma$ can be covered with open subsets, each of which is diffeomorphic onto a subset of $\mathbb{RP}^2$ in such a way that the geodesics of $\mathfrak{p}$ are mapped onto (segments) of projective lines $\mathbb{RP}^1\subset \mathbb{RP}^2$. This condition turns out to be equivalent to $\Sigma$ carrying an atlas modelled on $\mathbb{RP}^2$, that is, an atlas whose chart transitions are restrictions of fractional linear transformations. On the universal cover $\tilde{\Sigma}$ of the surface the charts can be adjusted to agree on overlaps, thus defining a~\textit{developing map} $\mathrm{dev} : \tilde{\Sigma} \to \mathbb{RP}^2$, unique up to post-composition with an element of $\mathrm{SL}(3,\R)$. In addition, one obtains a \textit{monodromy representation} $\rho :\pi_1(\Sigma) \to \mathrm{SL}(3,\R)$ of the fundamental group $\pi_1(\Sigma)$ -- well defined up to conjugation -- making $\mathrm{dev}$ into an equivariant map. A flat projective structure is called~\textit{properly convex} if $\mathrm{dev}$ is a diffeomorphism onto a subset of $\mathbb{RP}^2$ which is bounded and convex. If $\Sigma$ is a compact orientable surface with negative Euler characteristic, then (the conjugacy class of) `the' monodromy representation $\rho$ of a properly convex projective structure is an element in the Hitchin component $\mathcal{H}_3$ of $\Sigma$ and conversely every element in $\mathcal{H}_3$ can be obtained in this way~\cite{MR1145415}.

Motivated by the circle of ideas discussed in the introduction, it is shown in~\cite{MR2402597} and~\cite{MR1828223} that on a compact oriented surface $\Sigma$ of negative Euler characterstic, the convex projective structures are parametrised in terms of pairs $([g],C)$, consisting of a conformal structure $[g]$ and a cubic differential $C$ that is holomorphic with respect to the complex structure induced by $[g]$ and the orientation. Indeed, given a holomorphic cubic differential $C$ on such a $\Sigma$, there exists a unique Riemannian metric $g$ in the conformal equivalence class $[g]$, so that
\begin{equation}\label{eq:wangsequation}
K_g=-1+2|C|^2_g,
\end{equation}
where $K_g$ denotes the Gauss curvature of $g$ and $|C|_g$ the pointwise norm of $C$ with respect to the Hermitian metric induced by $g$ on the third power of the canonical bundle $K_{\Sigma}$ of $\Sigma$. Now there exists a unique section $\alpha$ of $K_{\Sigma}^2\otimes \ov{K_{\Sigma}^{*}}$, so that $\alpha\otimes d\mu_g=C$, where here we think of the area form $d\mu_g$ of $g$ as a section of $K_{\Sigma}\otimes \ov{K_{\Sigma}}$. Consequently, we obtain a connection $\nabla={}^{g}\nabla+2\Re(\alpha)$ on $T\Sigma$. The projective structure defined by $\nabla$ is properly convex and conversely every properly convex projective structure arises in this way~\cite[Theorem 4.1.1, Theorem 4.2.1]{MR2402597}. The metric $g$ is known as the~\textit{affine metric} or~\textit{Blaschke metric}, due to the fact that its pullback to the universal cover $\tilde{\Sigma}$ of $\Sigma$ can be realised via some immersion $\tilde{\Sigma} \to \mathbb{A}^3$ as a complete hyperbolic affine $2$-sphere in the affine $3$-space $\mathbb{A}^3$. In particular,~\eqref{eq:wangsequation} is known as Wang's equations in the affine sphere literature~\cite{MR1178538}. We refer the reader to the survey articles~\cite{MR3329885},~\cite{MR2743442} as well as~\cite{arXiv:1002.1767} for additional details. 

Calling a conformal structure $[g]$ on $(\Sigma,\mathfrak{p})$~\textit{closed}, if the associated connection $\varphi$ on $F^+_{[g]}$ induces a flat connection on $\Lambda^2(T^*\Sigma)$, we obtain a novel characterisation of properly convex projective structures among flat projective structures:   
\begin{thm}\label{thm:charconvex}
Let $(\Sigma,\mathfrak{p})$ be a compact oriented flat projective surface of negative Euler characteristic. Suppose $\mathfrak{p}$ is properly convex, then the conformal equivalence class of the Blaschke metric is closed and extremal for $\mathcal{E}_{\mathfrak{p}}$. Conversely, if $\mathcal{E}_{\mathfrak{p}}$ admits a closed extremal conformal structure $[g]$, then $\mathfrak{p}$ is properly convex and $[g]$ is the conformal equivalence class of the Blaschke metric of $\mathfrak{p}$.       
\end{thm}
\begin{rmk}
It would be interesting to know if flat projective surfaces $(\Sigma,\mathfrak{p})$  exist for which $\mathcal{E}_{\mathfrak{p}}$ admits an extremal conformal structure that is not closed. 
\end{rmk}
\begin{proof}[Proof of Theorem~\ref{thm:charconvex}]
Assume $\mathfrak{p}$ is properly convex and let $([g],C)$ be the associated pair. Let $g$ the Blaschke metric satisfying~\eqref{eq:wangsequation} and $\varphi$ the connection on $F^+_{[g]}$ induced by the Levi-Civita connection of $g$. Recall that $\nabla_{\varphi}$ denotes the connection induced by $\varphi$ on $TM$, hence here we have $\nabla_{\varphi}={}^g\nabla$. From~\cite{MR2402597} we know that $\mathfrak{p}$ is defined by a connection of the form
$$
\nabla={}^g\nabla+2\Re(\alpha), 
$$
where $\alpha$ satisfies $\alpha\otimes d \mu_g=C$. A simple computation shows that a torsion-free connection $\varphi$ on $F^+_{[g]}$ induces a flat connection on $\Lambda^2(T^*\Sigma)$ if and only if $\nabla_{\varphi}$ has symmetric Ricci tensor. Since here $\nabla_{\varphi}={}^g\nabla$ is a Levi-Civita connection, it follows that the conformal structure defined by the Blaschke metric is closed. In addition, since $C$ is holomorphic, we have $\nabla_{\varphi}^{\prime\prime} C=0$ and furthermore, since $d\mu_g$ is parallel with respect to ${}^{g}\nabla$, it follows that $\nabla_{\varphi}^{\prime\prime} \alpha$ vanishes identically, thus showing that the conformal structure defined by the Blaschke metric is extremal for $\mathcal{E}_{\mathfrak{p}}$. 

Conversely, let $(\Sigma,\mathfrak{p})$ be a compact oriented flat projective surface of negativ Euler characteristic. Suppose $[g]$ is a closed and extremal conformal structure for $\mathfrak{p}$. We let $\varphi$ denote the induced connection on $F^+_{[g]}$ and $\alpha$ the corresponding section of $K_{\Sigma}^2\otimes \ov{K_{\Sigma}^{*}}$.   Lemma~\ref{lem:struceqexcon} implies that on $P^{\prime}_{[g]}\simeq F^+_{[g]}$ we have the following structure equations, where we write $\omega$ instead of $\zeta_1$ 
\begin{align*}
\d a&=a^{\prime}\omega-q \ov{\omega}+2a\varphi-a\ov{\varphi},\\
\d k&=k^{\prime}\omega+k^{\prime\prime}\ov{\omega}+k\varphi+k\ov{\varphi},\\
\d q&=q^{\prime}\omega+\frac{1}{2}\left(\ov{L}+\ov{k^{\prime\prime}}-2\ov{q}a\right)\ov{\omega}+2q \varphi,\\
\label{eq:struceqforA}\d \varphi&=\left(|a|^2+\frac{1}{2}k-\ov{k}\right)\omega\wedge\ov{\omega}.
\end{align*}
Since $[g]$ is extremal, we know that $Q$ and hence $q$ vanishes identically. Moreover, recall that $\mathfrak{p}$ is flat if and only if $L\equiv 0$, hence the third structure equation gives
$$
0=\d q=q^{\prime}\omega+\frac{1}{2}\ov{k^{\prime\prime}}\ov{\omega}
$$
showing that the functions $q^{\prime}$ and $k^{\prime\prime}$ vanish identically as well. Lemma~\ref{lem:struceqcplxconriemsurf} implies that $-(\varphi+\ov{\varphi})$ is the connection form of the connection induced by $\varphi$ on $\Lambda^2(T^*\Sigma)$. Since $[g]$ is closed, the induced connection is flat and hence $\d(\varphi+\ov{\varphi})$ must vanish identically. Thus we obtain
$$
0=\d(\varphi+\ov{\varphi})=\frac{3}{2}\left(\ov{k}-k\right)\omega\wedge\ov{\omega},
$$
showing that $k$ must be real-valued. Note that since $k$ is real-valued, we have 
$$
0=\d(k-\ov{k})=k^{\prime}\omega-\ov{k^{\prime}}\ov{\omega},
$$
so that $k^{\prime}$ vanishes identically. Finally, we have reduced the structure equations to
\begin{align}
\d a&=a^{\prime}\omega+2a\varphi-a\ov{\varphi},\\
\label{eq:kappaparallel}\d k&=k\varphi+k\ov{\varphi},\\
\d \varphi&=\left(|a|^2-\frac{1}{2}k\right)\omega\wedge\ov{\omega}. 
\end{align} 
The equivariance property of the tautological $1$-form $\omega$ on $F^+_{[g]}$ gives
$$
(R_{r\e^{\i\phi}})^*\omega=\frac{1}{r}\e^{\i\phi}\omega
$$
for all $r \e^{\i\phi} \in \mathrm{CO}(2)$. The function $k$ represents a $(1,\! 1)$-form $\kappa$ on $\Sigma$ which satisfies $\upsilon^*\kappa=\frac{\i}{2}k\omega\wedge\ov{\omega}$. Consequently, $k$ has the equivariance property $(R_{r\e^{\i\phi}})^*k=r^2k$. Recall that 
$$
\int_{\Sigma} \i \d\varphi=\int_{\Sigma}\i\left(|a|^2-\frac{1}{2}k\right)\omega\wedge\ov{\omega}= 2\pi\chi(\Sigma)<0,
$$ 
hence $k$ must be positive somewhere. Note that~\eqref{eq:kappaparallel} shows that the $(1,\! 1)$-form $\kappa$ represented by $k$ is parallel with respect to $\varphi$. Consequently, $k$ cannot vanish. Since $\Sigma$ is assumed to be connected, the equivariance property of $k$ implies that the equation $k=1$ defines a reduction $F^+_g \subset F^+_{[g]}$ to an $\mathrm{SO}(2)$-subbundle which is the orthonormal coframe bundle of a unique representative metric $g \in [g]$. On $F^+_g$ we have
$$
0=\d k=\varphi+\ov{\varphi},
$$
showing that we may write $\varphi=\i\phi$ for a unique $1$-form $\phi$ on $F^+_g$. Of course, $\phi$ is the Levi-Civita connection form of $g$ and hence using $\omega=\omega^1+\i\omega^2$, we obtain the familiar structure equation for the Levi-Civita connection of an oriented Riemannian $2$-manifold
$$
\d \phi=-\left(-1+2|a|^2\right)\omega^1\wedge\omega^2.
$$
We may define a cubic differential $C$ by setting $C=\alpha\otimes d\mu_g$ and since the pullback to $F^+_g$ of the area form of $g$ is $\omega^1\wedge\omega^2$, we conclude the the cubic differential $C$ is holomorphic and represented by the function $a$. Since
$$
\d \phi=-K_g\omega^1\wedge\omega^2,
$$
where $K_g$ denotes the Gauss curvature of $g$, we have
$$
K_g=-1+2|C|^2_g,
$$
where we use that $\upsilon^*|C|^2_g=|c|^2$. It follows that $g$ is the Blaschke metric associated to the pair $([g],C)$ and hence $\mathfrak{p}$ is a properly convex projective structure.  
\end{proof}

\subsection{Concluding remarks}

\begin{rmk}
Let $G_0$ be a real split simple Lie group and $S(G_0)$ the associated symmetric space. For our purposes we may take $G_0=\mathrm{SL}(3,\R)$ so that $S(G_0)=\mathrm{SL}(3,\R)/\mathrm{SO}(3)$, but the following results hold in the more general case. Suppose $\Sigma$ is a compact oriented surface of negative Euler characteristic and $\rho : \pi_1(\Sigma) \to G_0$ a representation in the Hitchin component for $G_0$. By a theorem of Corlette~\cite{MR965220}, the choice of a conformal structure $[g]$ on $\Sigma$ determines a map $\psi : \tilde{\Sigma} \to S(G_0)$ which is equivariant with respect to $\rho$ and harmonic with respect to the Riemannian metric on $S(G_0)$ and the conformal structure on $\tilde{\Sigma}$ obtained by lifting $[g]$. Furthermore, this map is unique up to post-composition with an isometry of $S(G_0)$. The energy density of the map $\psi$ descends to define a $2$-form $e_{\rho}([g])\,d\mu_g$ on $\Sigma$ and hence one may define an energy functional~\cite{MR887285},~\cite{MR2482204}
$$
\mathcal{E}_{\rho}([g])=\int_{\Sigma} e_{\rho}([g])\,d\mu_g. 
$$
The energy $\mathcal{E}_{\rho}([g])$ turns out to only depend on the diffeotopy class of $[g]$ and thus defines an energy functional on Teichm\"uller space for every representation $\rho$ in the Hitchin component of $G_0$. The Hopf differential of the map $\psi$ yields a holomorphic quadratic differential which descends to $\Sigma$ as well and it is conjectured~\cite{MR2342806},~\cite{MR2482204}, that for every representation in the Hitchin component there exists a unique conformal structure on $\Sigma$ whose associated Hopf differential vanishes identically. For such a conformal structure the mapping $\psi$ is harmonic and conformal, hence minimal. In~\cite{MR3583351} Labourie proves the existence of a unique $\rho$-equivariant minimal mapping $\psi : \tilde{\Sigma} \to S(G_0)$ in the case where $G_0$ has rank two (the case $G_0=\mathrm{SL}(3,\R)$ was treated previously in~\cite{MR2402597}). Labourie also shows the existence of such a mapping without any assumption on the rank of $G_0$ in~\cite{MR2482204}. Moreover, in~\cite{MR3583351}, the energy bound  
$$
\mathcal{E}_{\rho}([g])\geqslant -2\pi \chi(\Sigma)
$$
is obtained, with equality if and only if $\rho$ is a~\textit{Fuchsian representation}. 

Given our results it is natural to expect a relation between $\mathcal{E}_{\rho}$ and our functional $\mathcal{E}_{\mathfrak{p}}$, where $\rho$ is an element in the $\mathrm{SL}(3,\R)$ Hitchin component and $\mathfrak{p}$ denotes its associated properly convex projective structure. However, relating the representation $\rho$ to its associated projective structure $\mathfrak{p}$ in a way that would allow to establish the expected relation proves to be quite difficult. This may be investigated elsewhere.
\end{rmk}

\begin{rmk}
Although we are currently unable to prove this, the previous remark suggests that in the case of a properly convex compact oriented projective surface $(\Sigma,\mathfrak{p})$ of negative Euler characteristic, the conformal equivalence class of the Blaschke metric  is in fact the unique critical point of $\mathcal{E}_{\mathfrak{p}}$. As a partial result towards this claim, it is shown in~\cite{arXiv:1804.04616} that if a properly convex compact oriented projective surface $(\Sigma,\mathfrak{p})$ of negative Euler characteristic admits a compatible Weyl connection, then $\mathfrak{p}$ arises from a hyperbolic metric. 
\end{rmk}

\begin{rmk}
In~\cite{MR3384876}, it is shown that for a compact oriented projective surface $(\Sigma,\mathfrak{p})$ of negative Euler characteristic the functional $\mathcal{E}_{\mathfrak{p}}$ admits at most one absolute minimiser $[g]$ (i.e.~a conformal structure $[g]$ such that $\mathcal{E}_{\mathfrak{p}}([g])=0$). 
\end{rmk}

\begin{rmk}
In~\cite{arXiv:1609.08033}, the author shows that properly convex projective surfaces arise from torsion-free connections on $T\Sigma$ that admit an interpretation as Lagrangian minimal surfaces. Some of their properties are studied in~\cite{arXiv:1706.03554}. It would be interesting to relate these minimal Lagrangian surfaces to the minimal mapping $\psi$ constructed in~\cite{MR2402597}.   
\end{rmk}

\begin{rmk}
We have seen that oriented projective structures admitting extremal conformal structures arise from pairs $(\varphi,\alpha)$ on a Riemann surface $(\Sigma,[g])$, where $\alpha$ satisfies $\nabla^{\prime\prime}_{\varphi}\alpha\equiv 0$. The torsion-free connection $\varphi$ on $F^+_{[g]}$ induces a holomorphic line bundle structure $\ov{\partial}_E$ on $E=K_{\Sigma}^2\otimes \ov{K^*_{\Sigma}}$ and conversely, it is easy see that for every choice of a holomorphic line bundle structure $\ov{\partial}_E$ on $E$ there exists a unique torsion-free connection $\varphi$ on $F^+_{[g]}$ inducing $\ov{\partial}_E$. Hence we may equivalently describe these projective structures in terms of a pair $(\ov{\partial}_E,\alpha)$ satisfying $\ov{\partial}_E\alpha\equiv 0$.  
\end{rmk}

\begin{rmk}
The so-called Einstein affine hypersurface structures introduced in~\cite{MR3137456} also provide examples of projective surfaces admitting an extremal conformal structure. 
\end{rmk}

\appendix

\section{A Gauss--Bonnet type identity}\label{Ap:GB}

As a by-product of our considerations, we obtain a Gauss--Bonnet type identity:  
\begin{thm}\label{main}
Let $(\Sigma,\mathfrak{p})$ a compact oriented projective surface. Then for every section $s : \Sigma \to (Y,\Omega_\mathfrak{p})$ we have
\begin{equation}\label{gaussbonnet}
\int_{\Sigma}s^*\Omega_\mathfrak{p}=2\pi\chi(\Sigma). 
\end{equation}
\end{thm}

\begin{proof}
Since $\pi : Y \to \Sigma$ admits smooth global sections, it follows that $\pi^* : H^k(\Sigma) \to H^k(Y)$ is injective. Note that by construction the fibres of the bundle $\pi : Y \to \Sigma$ are diffeomorphic to $\left(\R_2\rtimes \mathrm{GL}^+(2,\R)\right)/\mathrm{CO}(2)$ and hence diffeomorphic to $\R_2\times D^2$. In particular, the fibre is contractible, thus we have $H^2(Y)\simeq H^2(\Sigma)\simeq \R$ showing that $\pi^* : H^2(\Sigma) \to H^2(Y)$ is an isomorphism. It follows that any two sections of $Y \to \Sigma$ induce the same map on the second de Rham cohomology groups. It is therefore sufficient to construct a section $s : \Sigma \to Y$ for which~\eqref{gaussbonnet} holds. From the proof of the Lemma~\ref{lem:lift} we know that for every conformal structure $[g] : \Sigma \to Z$ there exists a lift $\widetilde{[g]} : \Sigma \to Y$ so that on the pullback bundle $\projb_{[g]}^{\prime}$ we have
$$
\zeta_2=2\ov{a}\,\ov{\zeta_1}, \qquad \zeta_3=k\zeta_1+2\ov{q}\,\ov{\zeta_1},
$$
Since
$$
\tau^*\Omega_\mathfrak{p}=-\frac{\i}{4}\left(\zeta_1\wedge\ov{\zeta_3}+\zeta_3\wedge\ov{\zeta_1}+\zeta_2\wedge\ov{\zeta_2}\right),
$$
computing as in Proposition~\ref{lem:dirichlet} and using the above identities for $\zeta_2,\zeta_3$ gives
\begin{align*}
\int_{\Sigma} \widetilde{[g]}^*\Omega_\mathfrak{p}&=-\frac{\i}{4}\int_{\Sigma}\left(k+\ov{k}-4|a|^2\right)\zeta_1\wedge\ov{\zeta_1}=\frac{\i}{2}\int_{\Sigma}\d\varphi-\d\ov{\varphi}=2\pi\chi(\Sigma).
\end{align*}
\end{proof}

\providecommand{\noopsort}[1]{}
\providecommand{\mr}[1]{\href{http://www.ams.org/mathscinet-getitem?mr=#1}{MR~#1}}
\providecommand{\zbl}[1]{\href{http://www.zentralblatt-math.org/zmath/en/search/?q=an:#1}{Zbl~#1}}
\providecommand{\arxiv}[1]{\href{http://www.arxiv.org/abs/#1}{arXiv:#1}}
\providecommand{\doi}[1]{\href{http://dx.doi.org/#1}{DOI}}
\providecommand{\MR}{\relax\ifhmode\unskip\space\fi MR }
% \MRhref is called by the amsart/book/proc definition of \MR.
\providecommand{\MRhref}[2]{%
  \href{http://www.ams.org/mathscinet-getitem?mr=#1}{#2}
}
\providecommand{\href}[2]{#2}

\end{document}